\documentclass[12pt]{amsart}

\usepackage{amsfonts, amssymb, amscd}
\usepackage{graphicx}
\usepackage{subcaption}
\usepackage{verbatim}
\usepackage{eucal}
\usepackage{amssymb}
\usepackage{mathrsfs}
\usepackage{graphicx}
\usepackage{psfrag}
\usepackage{cite}
\usepackage{upref}
\usepackage{amsrefs}
\usepackage[all]{xy}

\graphicspath{ {images/} }

\newcommand{\spa}{\operatorname{Span}}

\newcommand{\Zz}{\mathbb{Z}}
\newcommand{\Cc}{\mathbb{C}}
\newcommand{\Pp}{\mathbb{P}}
\newcommand{\Rr}{\mathbb{R}}
\newcommand{\Qq}{\mathbb{Q}}

\newcommand{\ww}{\mathbf{w}}

\newcommand{\uu}{\mathbf{u}}
\newcommand{\0}{\mathbf{0}}

\newcommand{\spec}{\operatorname{Spec}}

\newcommand{\Ver}{\operatorname{Vert}}
\newcommand{\Conv}{\operatorname{Conv}}
\newcommand{\Ann}{\operatorname{Ann}}

\newcommand{\Hom}{\operatorname{Hom}}
\newcommand{\rk}{\operatorname{rank}}

\newcommand{\Oo}{\mathcal{O}}

\newcommand{\Ll}{\mathcal{L}}

\newtheorem{lemma}{Lemma}[section]
\newtheorem{theorem}[lemma]{Theorem}

\newtheorem{proposition}[lemma]{Proposition}
\theoremstyle{definition}
\newtheorem{definition}[lemma]{Definition}

\newtheorem{conjecture}[lemma]{Conjecture}

\newtheorem{remark}[lemma]{Remark}
\theoremstyle{remark}
\newtheorem*{proof*}{Proof}
\numberwithin{equation}{section}

\begin{document}

\title[Birationality of toric multiple mirrors]{On the birationality of complete intersections associated to
nef-partitions}

\author{Zhan Li}
\address{Rutgers University, Department of Mathematics, 110 Frelinghuysen Rd.,
Piscataway \\ NJ \\ 08854 \\ USA} \email{lizhan@math.rutgers.edu}

\curraddr{Beijing International Center for Mathematical Research\\ Peking University, 
Beijing 100871, China} \email{lizhan@math.pku.edu.cn}


\begin{abstract}
We prove that generic complete intersections associated to multiple
mirror nef-partitions are all birational. This result solves a
conjecture by Batyrev and Nill in \cite{BN08} under some mild
assumptions.
\end{abstract}

\maketitle

\tableofcontents


\section{Introduction}\label{sec1}

Mirror symmetry was first discovered in string theory as a duality
between families of $3$-dimensional Calabi-Yau manifolds. Since its
discovery more than twenty years ago, it has drawn much attention
from physicists and mathematicians. Among the methods of
constructions of mirror pairs, Batyrev and Borisov used the complete
intersections in toric varieties \cite{Batyrev94, Borisov93};
Berglund, H\"{u}bsch and Krawitz used the finite quotients of hypersurfaces
in weighted projective space \cite{BH93, Kra10}; Gross and Siebert used the
toric degeneration of Calabi-Yau varieties to connect the
Strominger-Yau-Zaslow approach and the Batyrev-Borisov
approach \cite{Gross05, GS06}.

\medskip

The Batyrev-Borisov construction is one of the best understood settings
in mirror symmetry. Batyrev \cite{Batyrev94} used $\Delta$-regular
hypersurfaces in toric varieties associated to reflexive polytopes
as a way to construct a large set of mirror pairs. In this case, the
mirror pair consists of the family of $\Delta$-regular hypersurfaces
associated to a reflexive polytope and the family of
$\Delta$-regular hypersurfaces associated to its dual polytope.
Borisov \cite{Borisov93} generalized Batyrev's construction by
considering nef-partitions of reflexive polytopes. A nef-partition
of a reflexive polytope corresponds to a decomposition of the
boundary divisor into nef Cartier divisors. In this case, the mirror pairs
are constructed as the family of complete intersections associated
to a nef-partition and the family of complete intersections
associated to its dual nef-partition. These complete intersections
are Calabi-Yau varieties, and their string-theoretic Hodge numbers
behave as predicted by mirror symmetry \cite{BB96}.

\medskip

Compared to hypersurfaces, complete intersections associated to
nef-partitions are more complicated. In particular, they may exhibit
nontrivial multiple mirror phenomenon, i.e. two Calabi-Yau varieties
$X, \tilde{X}$ may have the same mirror $Y$ \cite{CK99} depending on a choice of nef-partition. If
this is the case, the homological mirror symmetry
conjecture \cite{Kontsevich94} implies that the derived categories
of coherent sheaves on $X,\tilde{X}$ are equivalent. Indeed,
according to the conjecture, the derived categories of $X,\tilde{X}$
are expected to be equivalent to the Fukaya categories of their
mirrors, which in this case are the same because $X,\tilde{X}$ are
multiple mirrors.

\medskip

Besides derived equivalence, Batyrev and Nill asked
whether toric multiple mirrors (of any dimension) in the setting
of the Batyrev-Borisov construction are birational (\cite{BN08} Question 5.2). We give an
affirmative answer to this question in Theorem \ref{main} under some
mild assumptions:

\medskip

\noindent{\bf Theorem.}~~ \it~ Let $X, \tilde{X}$ be toric multiple
mirrors and $D$ be the determinantal variety ($D$ is explained in
Section \ref{determinantal variety}), if they are all irreducible
with $\dim D = \dim X =\dim \tilde{X}$, then $X ,\tilde{X}$ are
birational. \rm

\medskip

We noticed that under certain restrictions, birationality of
multiple mirrors are established in the Berglund-H\"ubsch-Krawitz
setting first by Shoemaker (\cite{Shoemaker14}  Theorem 1), then generalized by
Kelly (\cite{Kelly13} Theorem 4.3) and simplified by Clarke (\cite{Clarke13} Corollary 3.7). In Givental's Fano/Landau-Ginzburg setting, Prince shows that for Fano complete intersections in toric varieties, certain Laurent polynomial multiple mirrors are related by a mutation (\cite{Pri}, c.f. \cite{CKP14} Theorem 5.1). A similar result also appears in \cite{HC15} Theorem 2.24.

\medskip

We briefly describe the content of each section:

\medskip

In Section \ref{sec2}, we fix the notation used throughout the
paper. We give necessary background on reflexive Gorenstein cones,
nef-partitions, and their relations. In Section \ref{sec3}, we
explain the combinatorial meaning of multiple mirrors, reformulate
the question of Batyrev and Nill by using reflexive Gorenstein
cones, discuss the motivation of this question and give an example
which motivates our proof. In Section \ref{sec4}, we give a proof of
the main result Theorem \ref{main}. We also discuss the necessity of
its assumptions. Section \ref{application} is devoted to an application
of our main theorem to an example of Calabrese and
Thomas (\cite{CT14} Section 4 Second example). In particular, we show how to
check the extra assumptions on the determinantal variety -- though
the assumption is indispensable, it is quite convenient to check
once the nef-partition is known. In the Appendix, we give
the definition of $\Delta$-regularity, discuss its properties and
establish the fact that generic complete intersection associated to
a nef-partition has $\Delta$-regularity. The results of this section are used in the proof of Theorem \ref{main}.

\medskip

{\bf Acknowledgements.} I would like to express my deep gratitude to my advisor Professor Lev
Borisov for his patient guidance and constant help. I
would like to thank Howard Nuer for useful discussions related to
the subject and Hemanth Saratchandran for reading the manuscript. Finally, I am grateful to anonymous referees for their many valuable suggestions.


\section{Background}\label{sec2}

\subsection{Definitions of Gorenstein cones and nef-partitions}\label{gorenstein nef}

We fix the following notations throughout the paper. Let $M \cong
\Zz^d$ be a lattice of rank $d$, and $N = \Hom_{\Zz}(M,\Zz)$ be its
dual lattice with pairing $\langle \cdot, \cdot \rangle : M \times N
\to \Zz$. Let $M_{\Rr} := M \otimes_\Zz \Rr$, and $N_{\Rr}:= N
\otimes_{\Zz} \Rr$ be the $\Rr$-linear extensions. The pairing
between $M,N$ can be extended to $\langle \cdot, \cdot \rangle :
M_{\Rr} \times N_{\Rr} \to \Rr$. Let $\overline{M}= \Zz^s \oplus M$
be the lattice extended from $M$, and $\overline{N}=(\Zz^s)^\vee \oplus N$
be its dual lattice with pairing:
\[\begin{split}
\overline{M} \times \overline{N} &\to \Zz\\
(a_1,\cdots,a_s;m) \times (b_1,\cdots,b_s;n) &\mapsto \sum_{i=1}^s
a_ib_i + \langle m,n
\rangle\ ,\\
\end{split}
\] where the integer $s$ should be obvious from the context.

\medskip

Our convention of notation for lattices is as follows: we always use $M$ (or $N$) to denote the lattice where \emph{polytopes} live; if a \emph{nef-partition} lives in lattice $M$ (or
$N$), then the corresponding \emph{reflexive Gorenstein cone} will live in
$\overline{M}$ (or $\overline{N}$); however, when talking about general \emph{cones} which
may not necessarily come from nef-partitions, we use $\overline M_1$ (or $\overline N_1$) to denote
their lattice.

\medskip

For a set $S \subset M_\Rr$, we use $\Conv(S) $ to denote its
\emph{convex hull}. If $\Delta \subset M_\Rr$ is a
lattice polytope (i.e. the convex hull of finite lattice points) with the origin $\0$ in its interior ,
then its \emph{dual polytope} is defined as
\[\Delta ^\vee:=\{y \in N_\Rr \mid \langle x, y \rangle \geq -1
~\forall x \in \Delta \}.
\] We use
$\Ver({\Delta})$ to denote the set of \emph{vertices} of a lattice
polytope $\Delta$, and $l(\Delta)$ to denote the set of its
\emph{lattice} points, i.e. $l(\Delta) = \Delta \cap M$.

\begin{definition}[\cite{Batyrev94} Definition 4.1.5]\label{nef}
Let $\Delta$ be a lattice polytope with the origin $\0 \in \Delta$ as
an interior point. If the dual polytope $\Delta^\vee$ is also a
lattice polytope, then $\Delta$ is called a \emph{reflexive polytope}.
\end{definition}

\begin{definition}[\cite{BB97} Definition 2.4]
A $d$-dimensional rational polyhedral cone $K \subset (\overline M_1)_\Rr$ is
called a \emph{Gorenstein cone}, if it is generated by lattice points which
are contained in an affine hyperplane $\{ x\in (\overline M_1)_\Rr\mid \langle
x,n \rangle =1 \}$ for some $n \in \overline N_1$.
\end{definition}

This $n$ is uniquely determined if $\dim K =\rk \overline M_1 $, and this is
the only case considered in this paper. We denote this unique element
by $\deg^\vee$, and call it the \emph{degree element}. By definition,
$\deg^\vee$ must live in $K^\vee \cap \overline N_1$, where $ K^\vee := \{y
\in (\overline N_1)_\Rr \mid \langle x, y \rangle \geq 0~ \forall x \in K\}$
is the \emph{dual cone} of $K$.

\medskip

In general, $K$ is a Gorenstein cone does not imply that $K^\vee$ is a
Gorenstein cone. However, if this is the case, we arrive at the
notion of reflexive Gorenstein cone.

\begin{definition}[\cite{BB97} Definition 2.6]
 A Gorenstein cone $K$ is called a \emph{reflexive Gorenstein cone} if $K^\vee$
  is also a Gorenstein cone. Let $\deg \in
K, \deg^\vee \in K^\vee$ be the degree elements in $K,K^\vee$
respectively, then $\langle \deg,\deg^\vee \rangle$ is called the
\emph{index} of this pair of reflexive Gorenstein cones.
\end{definition}

We will see in a moment how reflexive Gorenstein cones are related to
nef-partitions. Before doing this we briefly recall the
notion of a nef-partition. In the projective toric variety defined by
a reflexive polytope, a nef-partition is equivalent to a
decomposition of the boundary divisor into a summation of nef Cartier
divisors. On the other hand, there exists a purely combinatorial
definition of a nef-partition without invoking any toric variety
constructions. For simplicity, we use this combinatorial definition
here. The readers can find its equivalent form and its motivation in
Borisov's original paper \cite{Borisov93} (Definition 2.5).

\begin{definition}[\cite{KRS03} Proposition 3.2, \cite{BN08} Definition 3.1]\label{NefPartitionDef}
If the Minkowski sum of $s$ lattice polytopes $\sum_{i=1}^{s}
\Delta_i $ is a reflexive polytope, and the origin $\0 \in \Delta_i$
($0$ may not be an interior point) for each $i$, then $\{\Delta_i
\mid i=1,\dots,s \}$ is called a \emph{length $s$ nef-partition} of $\Conv(\cup_{i=1}^s \Delta_i)$.
\end{definition}

Nef-partitions arise in pairs (\cite{Borisov93} Proposition 3.4): if we fixed a
nef-partition $\{\Delta_i \mid i=1,\dots,s \}$  with $\Delta_i
\subset M_\Rr$, then there exists a dual nef-partition $\{\nabla_i
\mid i=1,\dots,s \}$ with $\nabla_i \subset N_\Rr$. The relations
between them are
\begin{equation}\label{eq: relation between nef-partition and its dual}
\begin{split}
 &(\sum_{i=1}^{s} \Delta_i)^\vee = \Conv (\bigcup_{i=1}^{s} \nabla_i)\\
 &(\sum_{i=1}^{s} \nabla_i)^\vee = \Conv (\bigcup_{i=1}^{s} \Delta_i).
\end{split}
\end{equation}
Furthermore, they satisfy the property
\begin{equation}\label{eq: multiple of nef-partitions}
\min\langle \Delta_i , \nabla_j \rangle \geq -\delta_{ij},
\end{equation}
where $\delta_{ij}$ is the Kronecker delta. Moreover, for all $w_j \in \Ver(\nabla_j)-\{0\}$, the minimum value can
be achieved, that is
\begin{equation}\label{eq: achieve -1}
\min_{x \in \Delta_i} \langle x, w_j\rangle = - \delta_{ij}.
\end{equation}

The following figures (\cite{BN08} Example 5.1) exhibit a length 2 nef-partition of the convex polytope 
$\Conv((1,1),(-1,1),(-1,0),(0,-1),(1,0))$ and its dual nef-partition:
\[\begin{array}{cc}
\includegraphics[width=2.3cm]{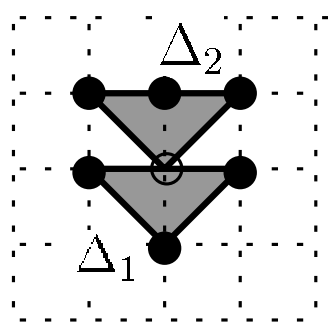} & \includegraphics[width=2.3cm]{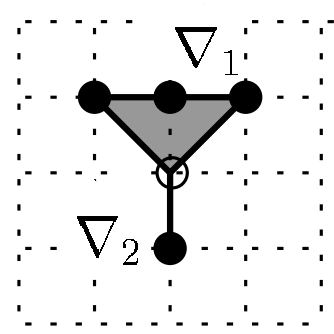}.\\
\end{array}
\] The nef-partitions may not be unique, in fact, another length 2 nef-partition of the same polytope can be obtained as follows
\[\begin{array}{cc}
\includegraphics[width=2.3cm]{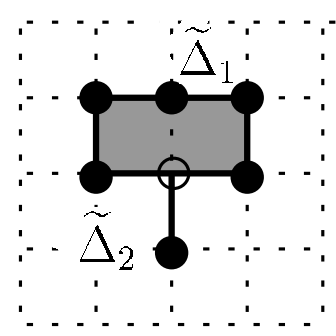} & \includegraphics[width=2.3cm]{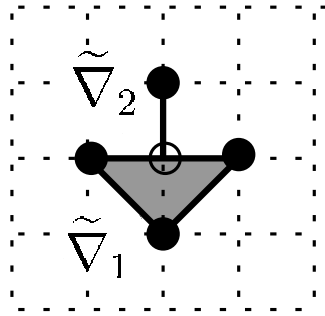}.\\
\end{array}
\]

\subsection{Nef-partitions versus reflexive Gorenstein cones}\label{nef-partition}
From a nef-partition, one can construct a reflexive Gorenstein
cone \cite{BB97}. On the other hand, for a reflexive Gorenstein
cone associated to a nef-partition,  if there exists a decomposition of
the degree element $\deg^\vee$, we can construct another
nef-partition. Now we will give a precise statement of the above
relations, which appeared in a slightly different form in
\cite{BN08}.

\medskip

Let $\overline M_1, \overline N_1$ be dual lattices. Let $K, K^\vee$ be full dimensional reflexive Gorenstein cones in
$(\overline M_1)_\Rr$, $(\overline N_1)_\Rr$ with degree elements $\deg, \deg^\vee$ respectively. Suppose the index is $\langle \deg,
\deg^\vee \rangle = s$ and
\[\deg^\vee = \sum_{i=1}^s e_i,\]
 with $e_i \in \overline N_1 \cap K^\vee $, $e_i
\neq 0$.

\medskip
Let
\begin{equation}\label{eq: some cones of K_1}
\begin{split}&R = \{x \in K \mid \langle x ,\deg^\vee \rangle = 1 \} \\
&R_i = \{x \in K \mid \langle x ,e_i \rangle = 1, \langle x ,e_j
\rangle =0, j \neq i\}\\
&T= \{y \in K^\vee \mid \langle \deg, y \rangle = 1 \}
\end{split}
\end{equation} be lattice polytopes. Because $K$ is a Gorenstein cone, any vertex $v$ of $R$ is a lattice
point. Thus $\langle v,e_i\rangle, 1 \leq i \leq s$ are nonnegative integers which
add up to $1$. Hence, there exists precisely one $e_i$ such that
$\langle v, e_i \rangle =1$. On the other hand, for any $e_j$,
because $e_j \neq 0$ and $K$ is a full dimensional cone, there
exists at least one vertex $w$ of $S$ such that $\langle w, e_j
\rangle =1$. Using these facts, one can show that $\{e_1, \dots,
e_s\}$ must be part of a $\Zz$-basis of $\overline N_1$.

\medskip

Let \[\Ann(e_1,\dots,e_s):=\{m \in \overline M_1 \mid \langle m,e_i \rangle
=0~ \forall i, 1 \leq i \leq s\}\] be a sublattice of $\overline M_1$ (we also denote it by $\Ann(e)$ for simplicity if no confusion arises), and
\[\spa_\Zz\{e_1,\dots,e_s\}:= \sum_{i=1}^s \Zz e_i \] be a sublattice of
$\overline N_1$. From the fact that $\{e_1,\dots,e_s\}$ is part of a
$\Zz$-basis, it follows that the pairing between $\overline M_1$ and $\overline N_1$ induces
a pairing
\[\Ann(e_1,\dots,e_s) \ \times \ \left(\overline N_1 / \spa_\Zz\{e_1,\dots,e_s\}\right)
\to \Zz,
\]
which identifies  $\Ann(e_1,\dots,e_s)$ and $\overline N_1 /
\spa_\Zz\{e_1,\dots,e_s\}$ as dual lattices.

\begin{proposition}[see \cite{BN08}]{\label{conetonef}}
Let $\sum_{i=1}^{s}R_i$ be the Minkowski sum of $R_i$, then the lattice polytope
\[\sum_{i=1}^{s}R_i -\deg \subset
\Ann(e_1,\dots,e_s)_{\Rr}\] is a reflexive polytope.
\end{proposition}

\begin{proof}
This follows from Corollary 2.5 and Theorem 2.6 ((1) $\Rightarrow$
(3)) in \cite{BN08}.
\end{proof}

Its converse is proved in \cite{BN08} Theorem 2.6 ((3) $\Rightarrow$
(1)):
\begin{proposition}[see \cite{BN08}]\label{neftocone}
Let $\Delta_1, \dots, \Delta_s \subset M_{\Rr}$ be lattice polytopes
such that the Minkowski sum $\sum_{i=1}^{s} \Delta_i $ has dimension
$\dim(M_\Rr)$ and $\sum_{i=1}^{s} \Delta_i -m $ is a reflexive
polytope for some $m \in M$. Let $\overline{M}=\Zz^s \oplus M$, then
the associated cone in $\overline{M}_\Rr$
$$K=\{(a_1,\dots,a_s; \sum_{i=1}^{s} a_i \Delta_i) \mid a_i \geq
0\}$$ is a reflexive Gorenstein cone of index $\langle \deg,
\deg^\vee \rangle=s$. In this case, $\deg = (1,\dots,1; m ) \in \overline M, \deg^\vee = (1, \dots, 1; \0) \in \overline N$.
\end{proposition}

Proposition \ref{neftocone} can be applied to the nef-partition $\{\Delta_i \mid 1 \leq i \leq s\}$,
where $\sum_{i=1}^s \Delta_i$ is itself a reflexive polytope with
dual polytope $(\sum_{i=1}^s\Delta_i)^\vee = \Conv(\cup_{i=1}^s
\nabla_i)$. The associated reflexive Gorenstein cone is
\begin{equation}\label{eq: associated ref cones: K}
K=\{(a_1,\dots,a_s;\sum_{i=1}^s a_i \Delta_i)  \subset
(\overline{M})_\Rr \mid a_i\geq 0\}.
\end{equation}
Its dual cone is
\begin{equation}\label{eq: associated ref cones: K^dual}
K^\vee=\{(b_1,\dots,b_s;\sum_{i=1}^s b_i \nabla_i) \subset
(\overline{N})_\Rr \mid b_i\geq 0\},
\end{equation} 
where $\{\nabla_i \mid 1 \leq i \leq s\}$ is the dual nef-partition. Moreover, the index of this pair of reflexive Gorenstein cones is exactly the same as the length of the nef-partition. A result of this form first appeared in \cite{BB97} Theorem 4.6.

\medskip

We will come back to the following setting several times in the sequel. To avoid repeating it each time, we will refer to it as $({\dagger})$:

\begin{quote}\label{dagger}
Let $\{\Delta_i \mid 1\leq i \leq
s\}$ be a nef-partition in $M$, and $\{\nabla_i \mid 1\leq i \leq
s\}$ be its dual nef-partition in $N$. Let $K \subset (\overline M)_\Rr, K^\vee \subset (\overline N)_\Rr$ be the associated reflexive Gorenstein cones as \eqref{eq: associated ref cones: K}, \eqref{eq: associated ref cones: K^dual} with degree elements $\deg, \deg^\vee$ respectively. 
There is a natural length $s$ decomposition
\begin{equation}\label{eq: decomposition of deg^dual}
\deg^\vee  = (1,1,\cdots,1;{\bf 0}) = \sum_{i=1}^s {e}_i,
\end{equation} with
\begin{equation}\label{eq: e_i}
{e}_i = (\underbrace{0,\dots,1,\dots,0}_{{\text{1 at the i-th
position}}};\0) \in \Zz^s \oplus N. 
\end{equation}

Suppose there exists another length $s$ decomposition
\begin{equation}\label{eq: decomposition of deg^dual}
\deg^\vee = (1,1,\cdots,1;{\bf 0}) = \sum_{i=1}^s \tilde{e}_i,
\end{equation} with $\tilde{e}_i \neq {\bf 0},
\tilde{e}_i \in K^\vee \cap \overline{N}$. 
Because the first $s$ coordinates of
$\tilde{e}_i$ are non-negative integers, without loss of generality,
we can assume
\begin{equation}\label{eq: tilde e_i}
\tilde{e}_i = (\underbrace{0,\dots,1,\dots,0}_{{\text{1 at the i-th
position}}};p_i) \in \Zz^s \oplus (N\cap \nabla_i). 
\end{equation}
\end{quote}

For each $\tilde e_i$, we define
$\tilde{R}_i$ as in \eqref{eq: some cones of K_1}. In this case,
$\left(\sum_{i=1}^s \tilde{R}_i - \deg\right)$ is a reflexive
polytope in $\Ann(\tilde{e}_1,\dots,\tilde{e}_s)$ by Proposition {\ref{conetonef}}. We claim that there exists a
lattice isomorphism, with origin mapping to origin,
\[
 \phi:\Ann(\tilde{e}_1,\dots,\tilde{e}_s) \to M
\]
defined by restricting to the projection $p: \Zz^s \oplus M \to M$.
In fact, if $\phi(x)={\bf 0}$, then $x = (a_1,\dots,a_s;{\bf 0})$, but $x \in
\Ann(\tilde{e}_1,\dots,\tilde{e}_s)$ implies that $\forall~ i,
a_i=0$, thus $\phi $ is injective. The surjectivity comes from the
fact that for $m\in M$, if we let $a_i = - \langle m,p_i\rangle$,
then $(a_1,\dots,a_s;m) \in \Ann(\tilde{e}_1,\dots,\tilde{e}_s)$
maps to $m$ under $\phi$.

\medskip

Under this isomorphism, we can identify
$\Ann(\tilde{e}_1,\dots,\tilde{e}_s)$ with $M$. Let
$\tilde{\Delta}_i=p(\tilde{R}_i)$, one can show that
\begin{equation}\label{eq: Delta}
\Conv(\bigcup_{i=1}^s \tilde{\Delta}_i)=\Conv(\bigcup_{i=1}^s {\Delta}_i)=:\Delta.
\end{equation}
Moreover, $(\sum_{i=1}^s\tilde{\Delta}_i)$ is a reflexive polytope in $M$ since $\phi(\deg) = {\bf 0}$ and $(\sum_{i=1}^s\tilde{R}_i-\deg)$ is a
reflexive polytope in $\Ann(e_1,\dots,e_s)$.
Because $e_i \in \tilde{R}_i$, we have ${\bf 0}\in \tilde{\Delta}_i$, and
this implies that $\{\tilde{\Delta}_i \mid 1\leq i \leq s\}$ is another
nef-partition of $\Conv(\cup_{i=1}^s \Delta_i)$ (see
Definition \ref{NefPartitionDef}).

\begin{remark}
One cannot exhaust $\emph{all}$ the nef-partitions of length $s$ of
$\Conv(\cup_{i=1}^s \Delta_i)$ using the above process (i.e. first
construct  reflexive Gorenstein cone $K,K^\vee$, then decompose
$\deg^\vee=\sum_{i=1}^s \tilde{e}_i$, and finally construct
$\tilde{\Delta}_i$). However, the above process will give exactly
the combinatorial data for toric multiple mirrors (see the details in
Theorem \ref{toric multiple mirror}).
\end{remark}

Next, we give the geometric meaning of this construction.

\medskip

Let $X(\Sigma(\Delta))$ be the toric variety defined by the fan
\begin{equation}\label{fan Delta}
\Sigma(\Delta):=\{{\bf 0}\}\cup\{\Rr_{\geq 0}\theta \mid \theta \subset
\Delta) {\ \rm is\  a\ face}\},
\end{equation} where $\Delta$ is defined in \eqref{eq: Delta}. Let
\[\Ll_i= \sum_{\rho \in \Ver(\Delta_i)\backslash \{0\}}
D_{\rho},\qquad \tilde{\Ll}_i= \sum_{\rho \in
\Ver(\tilde{\Delta}_i)\backslash \{0\}} D_{\rho}\] be the Weil
divisors corresponding to $\{\Delta_i \}, \{\tilde{\Delta}_i\}$
respectively, where $D_\rho$ is the torus invariant divisor
associated to the primitive element $\rho$. Then $\Ll_i, \tilde \Ll_i$ are nef Cartier divisors thanks to the nef-partition data (see \cite{Borisov93} Definition 2.5).
\medskip

In the sequel, we use $\chi^m \in \Cc[M]$ to denote the monomial associated to the lattice point $m\in M$. For example, if $m=(m_1, \cdots, m_d)$, then $\chi^m = \prod_{i=1}^d x^{m_i}_i$.

\begin{proposition}{\label{linearly equivalent of diviosrs}}
The nef-partition $\{\tilde{\Delta}_i \mid 1 \leq i \leq s\}$ of
$\Conv\left(\cup_{i=1}^s \Delta_i\right)$ is obtained from the
reflexive Gorenstein cone via the above procedure if and only if the corresponding divisors
$\{\tilde{\Ll}_i \mid 1 \leq i \leq s\}$ and $\{\Ll_i \mid 1 \leq i
\leq s\}$ are pairwise linearly equivalent.
\end{proposition}

\begin{proof}
Suppose $\deg^\vee = \sum_{i=1}^s \tilde{e}_i = \sum_{i=1}^s e_i$.
Without loss of generality, we can assume $\tilde{e}_i - e_i = p_i
\in N$ as in \eqref{eq: e_i}, \eqref{eq: tilde e_i}. Then one can check that $\tilde{\Ll}_i - \Ll_i$ is exactly
the principle divisor $(\chi^{p_i})$ on $X(\Sigma)$.

\medskip

On the other hand, suppose $\tilde{\Ll}_i, \Ll_i$ are linearly
equivalent divisors for each $i$, then there exists $p_i \in N$ such
that $\tilde{\Ll}_i - \Ll_i = (\chi^{p_i})$. One can check that
$\tilde{e}_i = e_i + ({\bf 0};p_i)$ satisfies the requirement.
\end{proof}

We will prove the birationality for the $\Delta$-regular complete
intersections associated to the above nef-partitions.


\section{The main question}\label{sec3}

\subsection{The main question and its motivation}\label{sec3.1}

After establishing the relation between reflexive Gorenstein cones
and nef-partitions, we are ready to state the question asked
in~\cite{BN08} more explicitly.

\medskip

Recall that in the setting of $(\dagger)$ of Section \ref{nef-partition}, we had shown how to associate a nef-partition $\{\tilde{\Delta}_i \mid 1 \leq i \leq s\}$ to the decomposition $\deg^\vee = \sum_{i=1}^{s}
\tilde{e}_i$.

\medskip

Whenever one has a polytope, there is a family of Laurent
polynomials associated to it. Let $l(\Delta_i)$ be the set of
lattice points in $\Delta_i$, then the family of Laurent polynomials
associated to $\Delta_i$ is
\begin{equation}\label{eq: Delta}
LP(\Delta_i):=\{f_i = \sum_{v \in l(\Delta_i)} c_v \chi^{v} \in \Cc[M]\mid c_v \in \Cc \},
\end{equation}
where $c_v \in \Cc$ is a complex coefficient associated to the vertex $v$. In
the same fashion, $\tilde{\Delta}_j$ produces a family of Laurent
polynomials
\begin{equation}\label{eq: tilde Delta}
LP(\tilde \Delta_j):=\{\tilde{f}_j = \sum_{v \in l(\tilde{\Delta}_j)} c_v \chi^{v} \in \Cc[M]\mid c_v \in \Cc \}.
\end{equation}

\begin{remark}\label{re: nef coefficient}
Given $s$ Laurent polynomials $f_i \in LP(\Delta_i), 1\leq i \leq s$, there is a natural way to identify them with another $s$ Laurent polynomials $\tilde f_j \in LP(\tilde \Delta_j), 1\leq j \leq s$. In fact, for nef-partitions, there exists the relation
\[\Ver(\Conv(\bigcup_{i=1}^s \Delta_i)) = \bigsqcup_{i=1}^s (\Ver(\Delta_i)\backslash \{\0\}) = \bigsqcup_{i=1}^s (\Ver(\tilde \Delta_i)\backslash \{\0\}).
\] We simply require that the coefficients $c_v$ are the \emph{same} for the \emph{same} non-origin vertex $v$. However, there is an indeterminacy to identify the coefficients associated to the origin ${\bf 0}$ (i.e. the constant terms). The way to identify them is not clear on the level of nef-partitions because all the origins are ``clustered'', it will be clear on the level of reflexive Gorenstein cones (c.f. Remark \ref{re: cone coefficient}) where the vertices corresponding to the origin are ``spread''.
\end{remark}

By saying a Laurent polynomial $f_i = \sum_{v \in l(\Delta_i)} c_v \chi^{v}$ has
\emph{general coefficients}, we mean that the coefficients $c_v$ are
chosen from a nonempty Zariski open subvariety of $\Cc^{\#{l(\Delta_i)}}$.

\medskip

The geometric meaning of $LP(\Delta_i)$ can be explained in terms of toric geometry. For the dual nef-partition $\{\nabla_i \mid 1 \leq
i \leq s\}$, we define
\[
\nabla:=\Conv(\bigcup_{i=1}^s \nabla_i)= (\sum_{i=1}^s
\Delta_i)^\vee
\] by \eqref{eq: relation between nef-partition and its dual}. Let 
\begin{equation}\label{eq: Sigma(nabla)}
\Sigma(\nabla)=\{\0\} \cup \{ \Rr_{\geq 0}\theta \mid \theta
\text{ is a face of } \nabla\}
\end{equation} be a fan similar to \eqref{fan Delta}, and $X(\Sigma(\nabla))$
be the toric variety defined by $\Sigma(\nabla)$ (it is the same variety as the projective toric variety
associated to the polytope $(\sum_{i=1}^s \Delta_i)$). We have nef torus
invariant Cartier divisors
\begin{equation}\label{eq: sheaf G}
{\mathcal{G}}_i = \sum_{\rho \in \Ver({\nabla_i})\backslash\{\0\}} D_\rho, \quad 1 \leq i \leq s. 
\end{equation} Notice that, in contrary to \eqref{fan Delta}, we are working in its dual lattice here.

\medskip

One can identify the global sections of $\mathcal{G}_i$ with Laurent
polynomials associated to $\Delta_i$ \cite{ToricVariety}:
\[
H^0(X(\Sigma(\nabla)), \mathcal{G}_i) \cong \{ \sum_{v \in l(\Delta_i)} c_v
\chi^v \mid c_v \in \Cc \} = LP(\Delta_i).
\]
Let $f_i \in  LP(\Delta_i)$, and let $V_{f_i}$
be the zero locus of $f_i$ with respect to $\mathcal{G}_i$ on $X(\Sigma(\nabla))$. Then
\begin{equation}\label{eq: overline X}
\{ \overline{X_{(\Delta_i)}} =\bigcap_{i=1}^s V_{f_i} \mid f_i \in H^0(X(\Sigma(\nabla)),
\mathcal{G}_i) \}
\end{equation} is a family of subschemes of $X(\Sigma(\nabla))$ parameterized
by the coefficients of $f_i, 1 \leq i \leq s$. In the Batyrev-Borisov construction, the mirror pairs are taken to be certain crepant partial desingularizations of $\overline{X_{(\Delta_i)}}$ and $\overline{X_{(\nabla_i)}}$. However, because we
focus on the birationality of multiple mirrors, it is enough and more convenient  to restrict it to the big torus.

\medskip

To be precise, let
$X_{(\Delta_i)} \subset (\Cc^{*})^d$ be
\[
X_{(\Delta_i)}:=\{  f_1 = f_2 =\cdots =f_s =0\}.
\]
Similarly, let $X_{(\tilde{\Delta}_i)} \subset (\Cc^{*})^d $ be
\[
X_{(\tilde{\Delta}_i)}:=\{ \tilde{f_1} = \tilde{f_2} =\cdots
=\tilde{f_s} =0\}.
\] Moreover, $\{f_i \mid 1 \leq i \leq s\}$ and $\{\tilde f_j \mid 1 \leq j \leq s\}$ are always identified as in Remark \ref{re: nef coefficient}. This gives an 1-1 correspondence between $\{X_{(\Delta_i)}\}$ and $\{X_{(\tilde\Delta_i)}\}$. In the sequel, we will always work under this identification, and write them as $X_{(\Delta_i)}$ and $X_{(\tilde\Delta_i)}$ without explicitly mentioning it.

\medskip

The following question was asked by Batyrev and Nill in \cite{BN08}
(Question 5.2):
\smallskip

({\textbf{Nef-partition version}})

\begin{quote}
Are the Calabi-Yau complete intersections $X_{(\Delta_i)}$ and
$X_{(\tilde{\Delta}_i)}$ birational to each other?
\end{quote}

\medskip

The physics relation between $X_{(\Delta_i)}$ and
$X_{(\tilde{\Delta}_i)}$ can be visualized as follows (c.f. Theorem \ref{toric multiple mirror}):
\begin{equation}\label{eq: fig multiple mirrors}
\xymatrix{
X_{(\Delta_i)} \ar@{<->}[dd]_{\rm multiple~ mirrors} \ar@{~>}[rrd]^{\rm ~mirror} \\
&&\qquad X_{(\nabla_i)} = X_{(\tilde \nabla_i)}\\
X_{(\tilde{\Delta}_i)} \ar@{~>}[rru]_{\rm ~mirror}
}
\end{equation}

\begin{remark}
Batyrev and Nill's question asks about the birationality of $X_{(\Delta_i)}$ and
$X_{(\tilde{\Delta}_i)}$ under the identification of coefficients as in Remark \ref{re: nef coefficient}. In general, neither two general elements $X''_{(\Delta_i)}, X''_{(\Delta_i)}$ in the same family $\{X_{(\Delta_i)}\}$ associated to a fixed nef-partition, nor a mirror pair $X_{(\Delta_i)}, X_{(\nabla_i)}$ are birational.
\end{remark}

We can reformulate this question in terms of reflexive Gorenstein
cones as well.

\medskip

Recall that in \eqref{eq: some cones of K_1}, we defined nonempty lattice polytopes
\begin{equation}\label{eq: R}
\begin{split}
&{R} = \{x \in K \mid \langle x ,\deg^\vee \rangle = 1\},\\
&{R}_i=\{v \in K \mid \langle v,\deg^\vee \rangle = \langle v, 
{e}_i \rangle =1\}.
\end{split}
\end{equation} 
Because $\deg^\vee = \sum_{i=1}^s
{e}_i$, for each lattice point $v$ in ${R}$, there exists a unique $i$, such that $\langle v,
{e}_i \rangle =1$. We have a disjoint union
$l({R})=\bigsqcup_{i=1}^s l({R}_i)$. One can define a family of
Laurent polynomials in $\Cc[\overline{M}]$ by setting:
\begin{equation}\label{eq: Laurent poly}
LP(e_i):=\{{g}_i = \sum_{v \in l({R}_i)} c_v \chi^{v}\mid c_v \in  \Cc\}.
\end{equation}
For any lattice point $ w_i$ such that $\langle  w_i, {e}_j
\rangle =\delta_{ij}$, $
\chi^{- w_i} \cdot {g}_i$ is a Laurent polynomial in
$\Cc[\Ann({e}_1,\dots,{e}_r)]$. We can similarly define
a family of intersections  $ X_{({e}_i)} \subset (\Cc^*)^d =
\spec(\Cc[\Ann({e}_1,\dots, {e}_s)])$ by
\begin{equation}\label{eq: X_e}
X_{({e}_i)}:=\{\chi^{- w_1} \cdot {g}_1= \chi^{- w_2} \cdot
{g}_2=\cdots=\chi^{- w_s} \cdot {g}_s=0\}
\end{equation}
This family does not depend on the choice of $ w_i$, because
any other choice will differ by a factor $\chi^{ w},  w \in
\Cc[\Ann({e}_1,\dots, {e}_s)] $ and this will not affect
the zero loci defined in $(\Cc^*)^d$.

\medskip

Similarly, we can construct $\tilde R_i$ and $LP(\tilde e_i):=\{\tilde g_i = \sum_{v \in l({\tilde R}_i)} c_v \chi^{v}\mid c_v \in  \Cc\}$ associated to the
decomposition $\deg^\vee = \sum_{i=1}^r \tilde e_i$. The corresponding family of intersections $X_{(\tilde e_i)} \subset (\Cc^*)^d = \spec(\Cc[\Ann({\tilde e}_1,\dots,{\tilde e}_s)])$ are given by
\begin{equation}\label{eq: tilde X_e}
X_{(\tilde e_i)}:=\{\chi^{-\tilde w_1}\cdot {\tilde g}_1=\chi^{-\tilde w_2}\cdot {\tilde g}_2=\cdots=
\chi^{-\tilde w_s}\cdot {\tilde g}_s=0\},
\end{equation} where $\tilde w_i$ satisfies $\langle \tilde w_i, \tilde e_j \rangle = \delta_{ij}$.

\medskip

We can compare these defining equations with those associated to the above nef-partitions. Because $ ({R}_i -\deg)$ is sent to $ {\Delta}_i$ (c.f. Section \ref{nef-partition}), we can identify
$ {g}_i \in \Cc[\overline M]$ with $f_i
\in \Cc[M]$ up to a factor $\chi^{m_i},m_i \in M$. Hence,
$X_{({e}_i)}$ and $X_{({\Delta}_i)}$ are indeed isomorphic
varieties. The same holds for $X_{(\tilde e_i)}$ and
$X_{(\tilde \Delta_i)}$ as well.

\begin{remark}\label{re: cone coefficient}
Similar to Remark \ref{re: nef coefficient}, there is an 1-1 correspondence between $\{g_i\in LP(e_i) \mid 1 \leq i \leq s\}$ and $\{\tilde g_i \in LP(\tilde e_i) \mid 1 \leq i \leq s\}$ by requiring that the coefficients $c_v$ corresponding to the lattice $v$ are the same in all the Laurent polynomials (provided $c_v\chi^v$ is a monomial of it). In fact, because $\bigsqcup_{i=1}^s l(R_i) = \bigsqcup_{i=1}^s l(\tilde R_i)$, if $c_v \chi^v$ is a monomial in $g_i$, then there must exist a unique $j$, such that $\tilde g_j$ also contains $c_v \chi^v$ as a monomial. Moreover, there is no indeterminacy to assign coefficients of the origin as in the nef-partition case.

\medskip

In the following, we only work under this identification, and write them as $X_{(e_i)}$ and $X_{(\tilde e_i)}$ without explicit mentioning it.
\end{remark}

\begin{theorem}\label{toric multiple mirror}
The complete intersections $X_{(\tilde{e}_i)}$ and $X_{(e_i)}$ are
toric multiple mirrors in the sense that they both mirror to the same
family.
\end{theorem}

\begin{proof}
By the toric mirror construction in \cite{Batyrev94, BB94},
the mirror of $X_{(\tilde{e}_i)}$ (or more precisely, a certain crepant partial desingularization of its compactification) is a family of generic complete
intersections defined by divisors $\{\tilde{\Ll}_i | 1 \leq i \leq
s\}$ in the toric variety $X(\Sigma(\Delta))$ (see \eqref{fan Delta}). Likewise, the
mirror of $X_{(e_i)}$ is a family of generic complete intersections
defined by divisors $\{\Ll_i | 1 \leq i \leq s\}$ in $X(\Sigma(\Delta))$. By
Proposition \ref{linearly equivalent of diviosrs}, $\{\tilde{\Ll}_i
| 1 \leq i \leq s\}$ and $\{\Ll_i | 1 \leq i \leq s\}$ consist of
pairwise linearly equivalent divisors. As a result, they define the
same family of complete intersections which is the mirror of both
$X_{(\tilde{e}_i)}$ and $ X_{(e_i)}$.
\end{proof}

Viewing the original question from this perspective (c.f. \eqref{eq: fig multiple mirrors}), we can ask:
\smallskip

(\textbf{Reflexive Gorenstein cone version})
\begin{quote}
Are the toric multiple mirrors $X_{(e_i)}, X_{(\tilde{e}_i)}$
birational?
\end{quote}

\medskip

We give an affirmative answer to this question in Theorem \ref{main}
under some mild technical assumptions.

\subsection{Example}

In this section, we will illustrate the basic idea of the proof by
an explicit example.

\medskip

Let $\{u_1, \dots, u_{15}\}$ be a basis of $\Zz^{15}$, and
consider the sublattice $\overline M_1 \subset \Zz^{15}$ defined as
\[
\overline M_1:= \{\sum_{i=1}^{15} l_i u_i \in \Zz^{15} \mid \sum_{i=1}^5 l_i
=\sum_{i=6}^{10} l_i =\sum_{i=11}^{15} l_i  \}.
\]
The rank of $\overline M_1$ is $13$, it contains a cone $K=\Zz_{\geq 0}^{15}
\cap \overline M_1$ which is defined by non negativity of all $l_i$. The $125$
generators of rays of $K$ are given by $u_{i_1} + u_{i_2} + u_{i_3}$
with $5j-4 \leq i_j \leq 5j$, and let $c_{ijk} \in \Cc$ denote
coefficients. Suppose $\{v_1,\dots,v_{15}\}$ is the dual basis of
$\{u_1,\dots,u_{15}\}$, then the dual lattice $\overline N_1$ is the
quotient of $\Zz^{15}$:
\[
\overline N_1 = \Zz^{15}/\spa_\Zz \{\sum_{i=1}^5v_i
-\sum_{i=6}^{10}v_i,\sum_{i=1}^5v_i -\sum_{i=11}^{15}v_i\}.
\]
The dual cone $K^\vee$ is the image of $\Zz_{\geq 0}^{15}$ in
$\overline N_1$, and its rays are generated by $v_i, 1 \leq i \leq 15$. The
degree elements $\deg, \deg^{\vee}$ are given by  $\sum_{i=1}^{15}
u_i$ and $\sum_{i=1}^5 v_i$ respectively.

\medskip

There are three different maximal ways of decomposing $\deg^{\vee}$ as a
summation of lattice points in $K^{\vee}$:
\[
\deg^{\vee} = \sum_{i=1}^{5} v_i,\quad \deg^{\vee} = \sum_{i=6}^{10}
v_i,\quad\deg^\vee = \sum_{i=11}^{15} v_i.
\]
This gives three different complete intersections in $\Pp^4 \times
\Pp^4$.

\medskip

For $\deg^{\vee} = \sum_{i=1}^{5} v_i$, the defining equations associated to this
decomposition can be expressed as
\begin{equation}\label{eq: 5 eqs}
\begin{split}
\sum_{1 \leq j, k \leq 5} c_{1jk} ~x_1 y_j z_k &=0\\
\sum_{1 \leq j, k \leq 5} c_{2jk} ~x_2 y_j z_k &=0\\
&\vdots\\
\sum_{1 \leq j, k \leq 5} c_{5jk} ~x_5 y_j z_k &=0\quad.
\end{split}
\end{equation}
Here $[x_1:x_2:\cdots:x_5] $ are homogenous coordinates of $\Pp^4$,
and similarly for $y_j, z_k$.

\medskip

As explained before, we can multiply each equation in \eqref{eq: 5 eqs} by a factor in
order to make it well defined in $\overline M_1 \cap \Ann(v_1,\dots,v_5)$.
Hence, let
\[
f_i(y,z) = x_i^{-1} \sum_{1 \leq j, k \leq 5} c_{ijk}~ x_i y_j z_k =
\sum_{1 \leq j, k \leq 5} c_{ijk}~ y_j z_k=0, \quad 1 \leq i \leq 5.
\]
This can be viewed as five bidegree $(1,1)$ equations in $\Pp^4
\times \Pp^4$. Similarly, for $\deg^\vee = \sum_{i=6}^{10} v_i$ and
$\deg^\vee = \sum_{i=11}^{15} v_i$ we have defining equations:
\[
\begin{split}
g_j(x,z)&= \sum_{1 \leq i, k \leq 5} c_{ijk}~ x_i z_k=0, \quad 1
\leq j
\leq 5,\\
h_k(x,y) &= \sum_{1 \leq i,j \leq 5} c_{ijk}~ x_iy_j =0, \quad 1
\leq k\leq 5.
\end{split}
\]
Our question thus becomes whether these three complete intersections
are birational for general choice of $c_{ijk}$.

\medskip

Let $X_1$ be the variety defined by $f_i = 0, 1 \leq i \leq 5$. Let
$A_{1}(z)$ be the $5 \times 5$ matrix
\[
A_{1}(z)= \left( \sum_{ k =1}^5 c_{ijk}z_k \right)_{ij}\quad,\quad 1
\leq i, j \leq 5
\]
then $f_i =0, 1 \leq i \leq 5$ can be written as a matrix equation
\[
A_{1}(z) \left(\begin{array}{c}
y_1\\
\vdots\\
y_5
\end{array}\right)=0.
\]
Because $[y_1: y_2: \cdots :y_5] \in \Pp^4$, one must have $\det
(A_1(z)) =0$. Let $D_1$ be the variety defined by $\det(A_{1}(z)) =0
$ in $\Pp^4$. For general coefficients, one can directly argue that $X_1$ and
$D_1$ are birational.

\medskip

Analogously, the variety $X_2$ defined by $g_j = 0 , 1 \leq j \leq 5$
can be written as
\[
\left(x_1,\cdots,x_5\right) A_2(z) = 0,
\]
where
\[
A_2(z)=\left(\sum_{k=1}^5 c_{ijk} z_k\right)_{ij} \quad,\quad 1 \leq
i, j \leq 5.
\]
Let $D_2$ be the variety defined by $\det(A_2(z))=0$. Then $X_2$ and $D_2$ are also birational. On the
other hand, $D_1$ and $D_2$ are the same varieties, and hence
$X_1,X_2$ are birational.

\section{The main theorem}{\label{sec4}}

\subsection{The $s = 2$ case:  a baby version of the main theorem}

To orient the reader, we sketch a proof in this simplest case. On the one hand, we do not have to deal with those involved combinatorical situations in Section \ref{sec: decomposition of lattices}; on the other hand, all the essential ingredients have already appeared in this case.

\medskip

Suppose we are in the setting of ($\dagger$) and the index of the reflexive Gorenstein cone is $s=2$. 

\begin{theorem}[The $s=2$ case of the Theorem \ref{main}]\label{baby version of main}
Let $X_{(e_i)}, X_{(\tilde{e}_i)}$ be toric multiple mirrors associated to $\deg^\vee = e_1 + e_2 = \tilde e_1 + \tilde e_2$ as in \eqref{eq: X_e}, \eqref{eq: tilde X_e}. Let $D$ be the determinantal variety as in \eqref{eq: D in s=2 case}. If $X_{(e_i)},
X_{(\tilde{e}_i)}, D$ are irreducible, then $X_{(e_i)}$ and
$X_{(\tilde{e}_i)}$ are birational.
\end{theorem}

\begin{proof}[Sketch of the Proof]
As in ($\dagger$), we can assume that $e_1 = (1,0;\0), e_2 = (0,1; \0), \tilde e_1 = (1,0; p_1), \tilde e_2 = (0,2; p_2)$ with $p_i \in \nabla_i$ and $p_1 + p_2=\0$. The dimension of the vector space $\spec_\Rr\{e_1, e_2, \tilde e_1, \tilde e_2\}$ spanned by $e_i, \tilde e_j$ is either $2$ or $3$. When the dimension equals to $2$, then $p_i = 0$, and thus $e_i = \tilde e_i$. The result automatically holds. When the dimension equals to $3$, by the property of nef-partition \eqref{eq: achieve -1} there exists an element $m \in M$ such that $\langle m, p_1 \rangle =-1$. Put the lattice points of $\overline M$ as
\begin{equation}\label{eq: u,w}
u_1 : = (1,0; \0),\quad u_2:=(0,1; -m ), \quad w := (0,0; m) ,
\end{equation} then they satisfy 
\begin{equation}\label{eq: properties of u,w}
\begin{split}
&\langle u_1, e_1 \rangle = \langle u_1, \tilde e_1 \rangle = 1,   \langle u_1, e_2 \rangle = \langle u_1, \tilde e_2 \rangle =0,\\ 
&\langle u_2, e_2 \rangle = \langle u_2, \tilde e_1 \rangle =1, \langle u_2, e_1 \rangle = \langle u_2, \tilde e_2 \rangle = 0,\\ 
&\langle w, e_1 \rangle =\langle w, e_2 \rangle =0,  \langle w, \tilde e_1 \rangle = -1, \langle w, \tilde e_2 \rangle =1.
\end{split}
\end{equation}

For $1 \leq i, j \leq 2$, let
\begin{equation}\label{eq: polytope S_i,j}
S_{i,j} = \{v \in K \mid \langle v, \deg^\vee \rangle =1, \langle v, e_i \rangle = \langle v ,\tilde e_j \rangle = 1\}
\end{equation} be lattice polytopes. The polytopes $R_i, \tilde R_j$ defined in \eqref{eq: some cones of K_1} can be decomposed as disjoint unions
\begin{equation}\label{eq: disjoint union}
l(R_{i}) = l(S_{i,1}) \sqcup l(S_{i,2}) \ {\rm ~and~}\  l(\tilde R_{j}) = l(S_{1,j}) \sqcup l(S_{2,j}).
\end{equation}

The a Laurent polynomial $g_i \in LP(e_i)$ defined in \eqref{eq: Laurent poly} can be written as
\[
{g}_i = \sum_{v \in l({R}_i)} c_v \chi^{v} = g_{i,1} + g_{i,2},
\] with $g_{i,j} = \sum_{v \in l({S}_{i,j})} c_v \chi^{v}$. Likewise, $\tilde g_j \in LP(\tilde e_j)$ can be written as $\tilde g_j = g_{1,j}+ g_{2,j}$. 

\medskip

Let's consider the matrix
\begin{equation}\label{eq: 2 by 2 matrix before nomalization}
\left( \begin{array}{cc}
g_{1,1} & g_{1,2}\\
g_{2,1} & g_{2,2}\\
\end{array} \right).
\end{equation} A priori, the entries live in $\Cc[\overline M]$, however, we can normalize them to make the entries live in $\Cc[{\Ann}(e, \tilde e)]$. For example, using the lattice points $u_1,u_2, w$ defined in \eqref{eq: u,w}, each entry in the following matrix lives in $\Cc[{\Ann}(e, \tilde e)]$
\begin{equation}\label{eq: 2 by 2 matrix after normalization}
\left( \begin{array}{cc}
\chi^{-u_1}g_{1,1} & \chi^{-u_1-w}g_{1,2}\\
\chi^{-u_2}g_{2,1} & \chi^{-u_2-w}g_{2,2}\\
\end{array} \right).
\end{equation} We use $y$ to denote the coordinates of $\Cc[{\Ann}(e, \tilde e)]$ and thus denote the matrix \eqref{eq: 2 by 2 matrix after normalization} as $A(y)$.

\medskip

Let $\ww = (1, \chi^w)^t$ be a $2 \times 1$ matrix with $\chi^w$ to be the coordinate of $\spec(\Cc[\Zz w])$, then \[A(y) \ww = \left(\chi^{-u_1}(g_{1,1}+g_{1,2}), \chi^{-u_2}(g_{2,1}+ g_{2,2})\right)^t = \left(\chi^{-u_1}g_1, \chi^{-u_2}g_{2}\right)^t,\] with $\chi^{-u_i}g_i \in \Cc[\Ann(e)]$. Hence $X_{(e_i)} = \{\chi^{-u_1}g_1 = \chi^{-u_2}g_2 =0\}$ (c.f. \eqref{eq: X_e}) is the same as $\{A(y)\ww = 0\}$. Because $\ww \neq \0$, the determinant $\det A(y)$ must equal zero. Let its zero locus be
\begin{equation}\label{eq: D in s=2 case}
D: = \{\det A(y) = 0\} \subset \spec(\Cc[\Ann(e, \tilde e)]).
\end{equation} For general coefficients, $D$ is of dimension $(\rk M -2)$. Moreover, there exists the decomposition $\Ann(e) = \Ann(e, \tilde e) \oplus \Zz w$. The natural projection 
\begin{equation}\label{eq: pi in case s=2}
\begin{split}
\pi: \spec(\Cc[\Ann(e)]) &\to \spec(\Cc[\Ann(e, \tilde e)])\\
(y, \chi^{w}) &\mapsto y
\end{split}
\end{equation} maps $X_{(e_i)}$ to $D$. The fibre over point the $y \in D$ is cut out by one linear equation in $\spec(\Cc[\Zz w]) \cong \Cc^*$. Hence it consists of either a single closed point or the whole $\spec(\Cc[\Zz w])$. Explicitly, if $\chi_1^w, \chi_2^w$ are the closed points on the fibre, for any $\lambda \in  \Cc$, such that $\lambda \chi_1^w + (1-\lambda) \chi_2^w$ is non-zero, then it also lies on the fibre. If the generic fibre had dimension $1$, then there exists a Zariski open set $V \subset \pi(X_{(e)})$, such that 
\begin{equation}\label{eq: birational}
\pi^{-1}(V) \to V \times \Cc^*
\end{equation} is birational. This will lead to a contradiction, and we leave the details to the proof of Theorem \ref{main}. Hence, 
$\pi$ is generically injective. 

\medskip

We have $\dim X_{(e)} = \dim D = \rk M - 2$, thus $X_{(e)}$ is birational to $D$ given both of them are irreducible. In the same fashion, we can show that $X_{(\tilde e)}$ is birational to $D$ under the irreducibility assumption. Eventually, $X_{(e)}, X_{(\tilde e)}$ are birational. 
\end{proof}

\begin{remark}\label{rmk: irreducibility}
Comparing with Theorem \ref{main}, we do not have to assume $\dim D$ equals $\dim X_{(e_i)}$ and $\dim X_{(\tilde e_i)}$ in Theorem \ref{baby version of main}. This is because in the $s=2$ case, $D$ is cut out by one equation and must have the expected dimension. However, we do have to assume $D$ to be irreducible. For example, when $S_{2,1} = \emptyset$, $D$
is the union of $\{\chi^{-u_1}g_{1,1}=0\}$ and $\{\chi^{-u_2-w}g_{2,2}=0\}$. By the proof of the
theorem, we see that for $(y, \chi^{w}) \in X_{(e_i)}$, if $(\chi^{-u_2-w}g_{2,2})(y) \neq 0$, then $\chi^w$ has to be zero. This leads to a contradiction, and thus $(\chi^{-u_2-w}g_{2,2})(y) = 0$. In particular, this shows that $X_{(e)}$ maps to the irreducible subvariety $\{\chi^{-u_2-w}g_{2,2} = 0\} \subset \spec(\Cc[\Ann(e, \tilde e)])$. By the above argument, we see that they are birational. A similar argument shows that $X_{(\tilde e)}$ is birational to $\{\chi^{-u_1}g_{1,1} = 0\}$.  A priori, the two loci are not expected to be birational.  Unfortunately, we do not have such examples on hand. To find a minimal set for such examples, one may have to go through the list of four dimensional reflexive polytopes and check all their length $2$ nef-partitions. Besides, it is meaningful to find such examples because in combination with the main result of \cite{FK14}, this will give derived equivalent but non-birational Calabi-Yau varieties.
\end{remark}

\subsection{Results on the decomposition of lattices}\label{sec: decomposition of lattices}
In order to work in the general setting, we need to consider various possibilities which do not appear in the length $2$ case. 
Let $\Delta$ be a reflexive polytope, $\{\Delta_i \mid 1\leq i \leq
s\}$ be a nef-partition of $\Delta$, and $\{\nabla_i \mid 1\leq i
\leq s\}$ be its dual nef-partition. In the following, we assume
$\dim \Delta= \dim M_\Rr$. Because $\Delta \subset \sum_{i=1}^{s}
\Delta_i$, we have $\dim (\sum_{i=1}^{s} \Delta_i ) = \dim
M_\Rr$.

\begin{lemma}\label{decomposition}
Let $p_i \in \nabla_i$, if $\sum_{i=1}^{s} p_i = 0$, and
\[\dim(\operatorname{Span}_\Rr \{ p_1,\dots,p_s\})=s-r,\] then there
exist disjoint sets $I_k, 1\leq k\leq r$, such that
$\bigsqcup_{k=1}^{r} I_k = \{1,\dots,s\}$ and for each $k$, we have
$\sum_{i \in I_k} p_i =0$.
\end{lemma}

\begin{proof}
Suppose $l$ is the maximum number such that there exists $l$
nonempty disjoint sets $I_j, 1 \leq j \leq l$ satisfying
\[I_1 \sqcup \cdots \sqcup I_l = \{1,\cdots,s\}\] and $\sum_{i \in I_j } p_i =0$ for all $j$.

\medskip

Because these $l$ equations are linearly independent, we have $$s-r
= \dim(\operatorname{Span}_\Rr \{ p_1,\dots,p_s\}) \leq s-l,$$ and
hence $l \leq r$. All we need to show is $l=r$.

\medskip

Suppose $l < r$, then there must exist at least one
equation $\sum_{ 1 \leq i  \leq s}a_i p_i=0$, which is not a linear
combination of $\sum_{i \in I_j} p_i = 0$. Hence, there must exist
an index $j$, such that for $i \in I_j, a_i$ are not identically the
same. Suppose $a_m$  is a minimal element in $\{a_i \mid i \in
I_j\}$. After reindexing the set, we can assume $j=1$ and $m=1$. Let
$C$ be a sufficiently large number, then
\[
 0=\sum_{1\leq i \leq s}a_i p_i - a_1\sum_{i\in I_1}p_i + C \cdot \sum_{i \in I_2 \sqcup \cdots \sqcup I_l }p_i = \sum_{2 \leq i \leq s}b_ip_i
\]
satisfies $b_i >0$ when $i \in I_2 \sqcup \cdots \sqcup I_l$, and
$b_i \geq 0$ when $ i \in I_1$. Moreover, there exists at least one
element $t \in I_1$ such that $b_t >0$ (because $a_i$ are not
identically the same for $i \in I_1$). Let $S=\{ i \mid b_i \neq
0\}$ be the index set corresponding to nonzero coefficients.

\medskip

Set $P = \sum_{i \in S}p_i = \sum_{i \in S}(1 - cb_i)p_i$ with $c$
sufficiently big such that for all $i, (1-cb_i) <0$. When $k \notin
S$, by \eqref{eq: multiple of nef-partitions} we have
\[
\begin{split}
&\langle \Delta_k, \sum_{i \in S}p_i \rangle \geq 0,\\
&\langle \Delta_k, \sum_{i \in S}(1-cb_i)p_i \rangle \leq 0.
\end{split}
\]
Hence $\langle \Delta_k , P \rangle = 0$ for $k \notin S$.

\medskip

In the following, we will show $P=0$. Otherwise, there exists $v \in
M_\Rr$ such that $\langle v , P \rangle >0$. Because $M_\Rr =
\sum_{i=1}^s \Rr_{\geq 0} \Delta_i$, we can choose $v = \sum_{1 \leq
i \leq s} v_i$ with $v_i \in \Delta_i$. Then we have
\[
 \langle v , P \rangle = \langle \sum_{i \in S} v_i + \sum_{i \notin S} v_i, P\rangle = \sum_{i \in S}
\langle v_i, -\sum_{j \notin S}p_j \rangle + \sum_{i \notin S}
\langle v_i,P \rangle.
\]
We use the assumption $\sum_{j=1}^s p_j =0$, and thus $P= - \sum_{j
\notin S}p_j$ in the second equation. However, $\sum_{i \in S}
\langle v_i, -\sum_{j \notin S}p_j \rangle  \leq 0$, and $\sum_{i
\notin S} \langle v_i,P \rangle = 0$ because $\langle \Delta_k , P
\rangle = 0$ for $k \notin S$. This contradiction implies $P=\sum_{i
\in S}p_i =0$.

\medskip

Now, as $I_1 \cap S \neq \emptyset$ and $I_1 \not\subseteq S$, the
index set $I_1':=I_1 \cap S$ must satisfy $\emptyset \subsetneqq
I_1' \subsetneqq I_1$. Since $I_2 \sqcup \cdots \sqcup I_l \subset
S$, we have
\[
 \sum_{j \in I_1'}p_j = P- \sum_{i \in I_2 \sqcup \cdots \sqcup I_l}p_i = 0.
\] But this implies \[\sum_{j \in I_1'}p_j = \sum_{j \in I_1 \backslash I_1'}p_j = 0\] which gives a further decomposition of $I_1$. This contradicts to the maximality of $l$.
\end{proof}

\begin{remark}
Using the notation of the lemma, we observe that for each $k$,
$\dim(\operatorname{Span}_\Rr\{p_i \mid i\in I_k\}) = \# (I_k) -1$.
\end{remark}

Another important fact for $\{p_i \mid 1 \leq i \leq s\}$ is that
their $\Zz$-span form a saturated sublattice in $N$, that is, the abelian group
$N/\left(\sum_{k=1}^s \Zz p_{i}\right)$ is torsion free. The
following combinatorial proof was suggested by Borisov.

\begin{lemma}\label{saturatedness}
The sublattice $\sum_{i=1}^s \Zz p_{i} \subset N$ is saturated.
\end{lemma}

\begin{proof}
Suppose otherwise, there exists $n=\sum_{i=1}^s a_i p_i$ with $a_i
\in \Qq$ such that $n \in N$ but $n \not\in \sum_{i=1}^s \Zz p_{i}$.
Furthermore, we can assume that for all $i, 0 \leq a_i <1$.

\medskip

Recall that $p_i \in \Delta_i$, hence $a_i p_i \in
\Delta_i$. By the property of nef-partition, we have
\[
n \in \sum_{i=1}^s  \Delta_i =  (\Conv(\bigcup_{i=1}^s
\nabla_i))^\vee.
\] If $n \neq \0$, then there exists a lattice point $m \in \Conv(\bigcup_{i=1}^s \nabla_i)$
such that $-1 \leq \langle n, m \rangle <0$. Because $n$ is a
lattice point, we have $\langle n, m \rangle = -1$.

\medskip

On the other hand, the set $\{ m \in  \Conv(\bigcup_{i=1}^s
\nabla_i) \mid \langle n, m \rangle =-1\}$ must contain some
vertices of $\Conv(\bigcup_{i=1}^s \nabla_i)$ and hence some
vertices of $\nabla_i$ because $\{\nabla_i \mid 1 \leq i \leq s\}$ is also a nef-partition. Without loss of
generality, we can assume $m \in \nabla_k$. By \eqref{eq: multiple of nef-partitions}, we have
\[
-1 = \langle n, m \rangle = \sum_{i=1}^s a_i \langle p_i, m \rangle
\geq -a_k > -1,
\] which is a contradiction. Thus $n = \0$, but this contradicts our
initial assumption on $n \not\in \sum_{i=1}^s \Zz p_{i} $.
\end{proof}

These results can be applied to ($\dagger$) where we have
\[ \tilde{e}_i =
(\underbrace{0,\dots,1,\dots,0}_{\text{1 at the i-th
position}};p_i), ~ p_i \in N\cap\nabla_i.\]
Moreover, the dimensions satisfy
\[\dim(\spa_{\Rr}\{e_1,\dots,e_s,\tilde{e}_1,\dots,\tilde{e}_s\}) = s
+ \dim(\spa_{\Rr}\{p_1,\dots,p_s\}).\] If
$\dim(\spa_{\Rr}\{p_1,\dots,p_s\})=s-r$, by
Lemma~\ref{decomposition} there exist disjoint index sets $I_k, 1
\leq k \leq r$, such that $\bigsqcup_{k=1}^{r} I_k = \{1,\dots,s\}$.
For each $k$, we have $\sum_{i \in I_k} p_i =0$, with
$\dim(\operatorname{Span}_\Rr\{p_i \mid i\in I_k\}) = \# (I_k) -1$. From now on, we set $n_k = \#(I_k)$.

\medskip

For our convenience, we use a superscript $(-)^{(k)}$ to reindex the
index set $I_k$ and the corresponding elements. For
example
\begin{equation}
\sum_{i \in I_k} p_i=0 
\end{equation}  
becomes
\begin{equation}
\sum_{i=1}^{n_k}
p^{(k)}_i=0
\end{equation} under the new notation.

\medskip

Because $\dim (\spa_{\Rr}\{p^{(k)}_1,\dots,p^{(k)}_{n_k}\})= n_k -1$, we
can choose
\[\{p^{(1)}_{2},\dots,p^{(1)}_{n_{1}},\dots,p^{(r)}_{2},\dots,p^{(r)}_{n_{r}}\}\] as a
$\Rr$-linearly independent set. 

\medskip

For simplicity, we denote $\Ann(e_1,\dots,e_s)$ and  $\Ann(e_1,\dots,e_s,\tilde{e}_1,\dots,\tilde{e}_s)$ by $\Ann(e)$ and $\Ann(e, \tilde e)$ respectively. 
Then $\Ann(e)$ is a sublattice of $ \overline{M}$
with the same rank as $M$, and
$\Ann(e, \tilde e)$ a sublattice of $ \overline{M}$ with rank $\rk (M)+r-s$.

\begin{lemma}\label{decomposelattice}
The lattice $\Ann(e) \subset \overline{M}$ can be decomposed as
\begin{equation}\label{eq: decomposition of Ann(e)}
\begin{split}
\Ann(e)=&\Ann(e,\tilde{e}) \\
&\oplus \Zz[w^{(1)}_{2}]\oplus \dots \oplus \Zz[w^{(1)}_{n_{1}}] \\
&\oplus \cdots \\
&\oplus \Zz[w^{(r)}_{2}] \oplus \dots \oplus \Zz[w^{(r)}_{n_r}],
\end{split}
\end{equation}
where $w^{(k)}_{i} \in \overline{M}$ satisfies the following requirements
(notice that by our indexing, $w^{(k)}_{i}$ starts from $w^{(k)}_{2}$):
\begin{enumerate}
\item $\langle w^{(k)}_{i}, \tilde{e}^{(k)}_{1}\rangle =-1,~ \langle w^{(k)}_{i}, \tilde{e}^{(k)}_{i}\rangle
=1$ for $i \geq 2$.
\item $\langle w^{(k)}_{i},\tilde{e}^{(l)}_{j} \rangle =0$ for all $\tilde{e}^{(l)}_{j} \neq \tilde{e}^{(k)}_{1}, \tilde{e}^{(k)}_{i}$.
\item $\langle w^{(k)}_{i}, e^{(l)}_{j}\rangle =0$ for all $e^{(l)}_{j}$.
\end{enumerate}

\end{lemma}

\begin{proof}
First,  if we already have $w_{ki}$ satisfying the given properties,
then by definition, we have
\[
\Ann(e,\tilde{e}) \oplus \Zz[w^{(1)}_{2}]\oplus \dots \oplus
\Zz[w^{(1)}_{n_{1}}] \oplus \dots \oplus \Zz[w^{(r)}_{2}] \oplus \dots \oplus
\Zz[w^{(r)}_{n_r}] \subset \Ann(e)
\]
as a sublattice. On the other hand, for any $m \in \Ann(e)$, we set
\[m- \sum_{k=1}^r \sum_{ i = 2}^{n_k}\langle m, \tilde{e}^{(k)}_{i}\rangle w^{(k)}_{i},\]
then by definition, one can check
\[
\begin{split}
&m- \sum_{k=1}^r \sum_{ i = 2}^{n_k}\langle m, \tilde{e}^{(k)}_{i}\rangle w^{(k)}_{i}\\
 &\in \Ann(e) \cap \Ann(\tilde{e}^{(1)}_{2},\cdots,\tilde{e}^{(1)}_{n_1},\cdots,\tilde{e}^{(r)}_{2},\cdots,\tilde{e}^{(r)}_{n_r}).\end{split}
\] By Lemma \ref{decomposition}, $\sum_{i=1}^{n_k}e^{(k)}_i = \sum_{i=1}^{n_k} \tilde{e}^{(k)}_i $ for all $k$, then 
\[m- \sum_{k=1}^r \sum_{ i = 2}^{n_k}\langle m, \tilde{e}^{(k)}_{i}\rangle w^{(k)}_{i} \in \Ann(e,\tilde{e}).
\] Thus, we only need to show the existence of $w^{(k)}_{i}$. 

\medskip

Let the lattice
map
\[
\theta : M \to \Zz^{s-r}
\]
be defined by
\[
m \mapsto (\langle m , p^{(1)}_{2}\rangle,\dots,\langle m ,
p^{(1)}_{n_1}\rangle,\dots,\langle m , p^{(r)}_{2}\rangle,\dots,\langle m ,
p^{(r)}_{n_{r}}\rangle).
\]
We claim that $\theta$ is a surjective lattice map. In fact, by the
saturatedness (Lemma \ref{saturatedness}),
\[
\{p^{(1)}_{2},\dots,p^{(1)}_{n_{1}},\dots,p^{(r)}_{2},\dots,p^{(r)}_{n_{r}}\}
\] forms
part of a $\Zz$-basis of $N$. Hence $\theta$ is surjective.

\medskip

We can choose $m$ such that $\langle m , p^{(i)}_{j} \rangle =0$ for all
$j \geq 2$ except $\langle m, p^{(k)}_{i} \rangle =1$, and set
\[w^{(k)}_{i} = (0,\dots,0;m) \in \overline{M},\]  then $w^{(k)}_{i}$ satisfies the required
properties.
\end{proof}

Now let
\begin{equation}\label{eq: L}
L = \spa_\Zz
\{w^{(1)}_{2},\cdots,w^{(1)}_{n_1},\cdots,w^{(r)}_{2},\cdots,w^{(r)}_{n_r}\} \subset
\overline{M}.
\end{equation}  We have $\Ann(e) =  \Ann(e, \tilde{e}) \oplus L$.

\medskip

Correspondingly, there exists a decomposition of tori:
\begin{equation}\label{eq: decomposition of toric}
\spec (\Cc[\Ann(e)]) = \spec (\Cc[\Ann(e,\tilde{e})]) \times
\spec(\Cc[L]).
\end{equation}

For any closed point in $\spec (\Cc[\Ann(e)])$ with coordinate $x$ ,
we will write $x=(y,\omega)$ with $y \in
\spec(\Cc[\Ann(e,\tilde{e})]),\ \omega \in \spec(\Cc[L])$
respectively.

\subsection{Construction of the determinantal variety}\label{determinantal
variety}

One main ingredient in the proof of previous theorem is  a
determinantal variety $D$ which serves as a bridge to connect two
complete intersections. We will show how to construct its analogue
in $\spec(\Cc[\Ann(e,\tilde{e})])$ which heavily relies on
Lemma \ref{decomposelattice}.

\medskip

Now we choose $u^{(k)}_{i} \in \overline{M}$ satisfying:
\begin{equation}\label{eq: def of u}
\begin{split}
&{\rm (1)~}\langle u^{(k)}_{i} , e^{(k)}_{i} \rangle = \langle u^{(k)}_{i},\tilde{e}^{(k)}_{1}
\rangle=1\\
&{\rm (2)~}\langle u^{(k)}_{i} , e^{(l)}_{j} \rangle =0  {\rm~for~ all~}  e^{(l)}_{j} \neq e^{(k)}_{i}\\
&{\rm (3)~}\langle u^{(k)}_{i} , \tilde{e}^{(l)}_{j} \rangle =0  {\rm~for~ all~}  \tilde{e}^{(l)}_{j} \neq \tilde{e}^{(k)}_{1}.
\end{split}
\end{equation}

We should point out that unlike those $w^{(k)}_{i}$ constructed before, $u^{(k)}_{i}$
starts from $u^{(k)}_{1}$ for each $k$. The existence of $u^{(k)}_{i}$ follows
from the same reasoning as  in Lemma \ref{decomposelattice}, and
we do not repeat it here.

\begin{lemma}\label{le: diagonal matrix}
The set $\{v \in K \mid \langle v,\deg^\vee
\rangle =1, \langle v,e^{(k)}_i\rangle = \langle v, \tilde{e}^{(l)}_j \rangle
=1\}$ is empty unless $k=l$. 
\end{lemma}
\begin{proof}
In fact, because of the relation $\sum_{i \in I_k} e_i=\sum_{i \in I_k}
\tilde{e}_i$ (Lemma \ref{decomposition}), for any $v$ such that $\langle v, \deg^\vee \rangle =1$, $\langle v ,\sum_{i \in
I_k} e_i \rangle =1 $ implies that $\langle v ,\sum_{i \in I_k}
\tilde{e}_i \rangle =1 $. Thus for any $l \neq k$, $\langle v, \tilde e^{(l)}_j \rangle = 0$. 
\end{proof}

Similar to \eqref{eq: polytope S_i,j}, we define lattice polytopes
\[S^{(k)}_{i,j} = \{v \in K \mid \langle v,\deg^\vee
\rangle =1, \langle v,e^{(k)}_i\rangle = \langle v, \tilde{e}^{(k)}_j \rangle
=1\}.\] Moreover, let 
\[\begin{split}
&R^{(k)}_i = \{v \in K \mid \langle v, \deg^\vee \rangle =1, \langle v, e^{(k)}_i \rangle =1\},\\
&\tilde R^{(k)}_j = \{v \in K \mid \langle v, \deg^\vee \rangle =1, \langle v, \tilde e^{(k)}_j \rangle =1\}.
\end{split}\]
Then by Lemma \ref{le: diagonal matrix}, we have disjoint
unions
\begin{equation}\label{eq: decompostion}
\begin{split}
&l(R^{(k)}_i)= \bigsqcup _{1 \leq j \leq n_k } l(S^{(k)}_{i,j}),\\
&l(\tilde R^{(k)}_j)= \bigsqcup _{1 \leq i \leq n_k } l(S^{(k)}_{i,j}).
\end{split}
\end{equation}

On the level of Laurent polynomials, we define
\[g^{(k)}_{i,j}=\sum_{v \in l(S^{(k)}_{i,j})} c_v \chi^v\] and
\[
g^{(k)}_i = \sum_{v \in l(R^{(k)}_i)} c_v \chi^v, \quad 
\tilde g^{(k)}_j = \sum_{v \in l(\tilde R^{(k)}_j)} c_v \chi^v
\]where $c_v \in \Cc$ are general coefficients as explained in Remark \ref{re: cone coefficient}. By \eqref{eq: decompostion}, we have
\begin{equation}\label{eq: decomposition of Laurent poly}
g^{(k)}_i = \sum_{1 \leq j \leq n_k} g^{(k)}_{i,j} \quad
{\rm and}\quad \tilde g^{(k)}_j = \sum_{1 \leq i \leq n_k} g^{(k)}_{i,j}.
\end{equation}

With these preliminaries, let $A^{(k)}(y)$ be the $n_k \times n_k$ matrix with entries in
$\Cc[M]$
\begin{equation}\label{eq: matrix A}
\begin{split}
A^{(k)}(y) =& \left( \begin{array}{cccc}
\chi^{-u^{(k)}_{1}} g^{(k)}_{1,1}& \chi^{-u^{(k)}_{1} - w^{(k)}_{2}} g^{(k)}_{1,2}&\cdots& \chi^{-u^{(k)}_{1} - w^{(k)}_{ n_k}} g^{(k)}_{1,n_k}\\
\chi^{-u^{(k)}_{2}} g^{(k)}_{2,1}& \chi^{-u^{(k)}_{2} - w^{(k)}_{2}} g^{(k)}_{2,2}&\cdots& \chi^{-u^{(k)}_{2} - w_{ n_k}} g^{(k)}_{2,n_k}\\
\vdots & \vdots & &\vdots \\
\chi^{-u^{(k)}_{n_k}} g^{(k)}_{n_k,1}& \chi^{-u^{(k)}_{n_k} - w^{(k)}_{2}} g^{(k)}_{n_k,2}&\cdots& \chi^{-u^{(k)}_{n_k} - w^{(k)}_{ n_k}} g^{(k)}_{n_k,n_k}\\
\end{array} \right)
\end{split}
\end{equation}

This is a generalization of the matrix \eqref{eq: 2 by 2 matrix after normalization}. Notice that the first column is not constructed identically as the rest. As one can directly verify, each entry of $A^{(k)}(y)$ lives in $\Cc[\Ann(e, \tilde e)]$, Hence by the
decomposition $\spec (\Cc[\Ann(e)]) = \spec
(\Cc[\Ann(e,\tilde{e})]) \times \spec (\Cc[L])$ (c.f.  \eqref{eq: decomposition of Ann(e)}), we use $y$ to
represent the corresponding coordinates in $\spec(
\Cc[\Ann(e,\tilde{e})])$.

\medskip

Next, we define a $ n_k \times 1$ matrix
\[
\ww^{(k)} = \left(1,\chi^{w^{(k)}_{2}},\cdots, \chi^{w^{(k)}_{n_{k}}}\right)^t,
\]and a $1 \times n_k$ matrix
\[
 \uu^{(k)} =\left(\chi^{u^{(k)}_{1}},\chi^{u^{(k)}_{2}},\dots,\chi^{u^{(k)}_{n_k}}\right).
\] By \eqref{eq: decomposition of Laurent poly}, the equation
\[
A^{(k)}(y)\cdot\ww_k = 0
\]
is exactly the same as
\[
\left( \begin{array}{c}

\chi^{-u^{(k)}_{1}}g^{(k)}_{1}\\
\vdots\\
\chi^{-u^{(k)}_{n_k}}g^{(k)}_{n_k}
\end{array}\right)=0.
\]

\medskip

By the same reason,
\[
\uu^{(k)} \cdot A^{(k)}(y) = 0
\]
is exactly the same as
\[
\left(\tilde g_1^{(k)}, \chi^{-w^{(k)}_{2}}\tilde{g}^{(k)}_{2},
\cdots,\chi^{-w^{(k)}_{n_k}}\tilde{g}^{(k)}_{n_k}\right) =0.
\]

Let \[D^{(k)} : = \{\det (A^{(k)}(y)) = 0\} \subset
\spec(\Cc[\Ann(e,\tilde{e})]),\] and $D$ be the intersection of $D^{(k)}$
\begin{equation}\label{eq: D}
D = \bigcap_{k=1}^r D^{(k)}
\end{equation}
with its reduced induced subscheme structure. This $D$ will be called \emph{determinantal variety} in the sequel. It will serve as
a bridge in the proof the birationality of two complete intersections.

\begin{remark}{\label{diagnoal not zero}}
We have $\det (A^{(k)}(y)) \not\equiv0$ because
$e^{(k)}_i \in S^{(k)}_{i,i}$. In fact, by choosing generic coefficients, these elements give nonzero summand in $\det (A^{(k)}(y))$. Hence, $D^{(k)}$ is a
hypersurface in $\spec (\Cc[\Ann(e,\tilde{e})])$ with dimension $\rk (M)
+r-s -1$.
\end{remark}

We state without proof the following lemma.

\begin{lemma}\label{birationality}
Let $f : X\to Y$ be a dominant morphism of varieties over $\mathbb
C$. Suppose $[K(X): K(Y)]=n$. Then there exists a dense open subset
$U$ of $Y$ such that $f^{-1}(y)$ consists of $n$ (distinct) points
for all $y\in U$. In particular, if $f$ is a dominant, generically injective morphism, then
$[K(X): K(Y)]=1$, and thus $X, Y$ are birational.
\end{lemma}

\subsection{Proof of the main theorem}

In this section, we will show that under some mild assumptions, $X_{(e_i)}$ and
$X_{(\tilde{e}_i)}$ are both birational to the determinantal variety
$D$ (c.f. \eqref{eq: D}), and hence they are birational to each other. In fact, we will show that the
morphism $X_{(e_i)}$ to $D$ induced by the projection from $\spec
(\Cc[\Ann(e)])$ to $\spec(\Cc[\Ann(e,\tilde{e})])$ gives the
birational morphism, and similarly for $X_{(\tilde{e}_i)}$ to $D$.

\medskip

Our setup is the same as $(\dagger)$ in Section \ref{nef-partition}. The general complete intersections (i.e. the toric multiple mirrors, see Theorem \ref{toric multiple mirror})
 $X_{(e_i)}, X_{(\tilde{e}_i)}$ are defined in Section \eqref{eq: X_e}, \eqref{eq: tilde X_e} and the determinantal variety $D$ is defined in \eqref{eq: D}. We have the following main theorem.

\begin{theorem}\label{main}
Let $X_{(e_i)}, X_{(\tilde{e}_i)}$ be toric multiple mirrors,
and $D$ be the determinantal variety. If $X_{(e_i)},
X_{(\tilde{e}_i)}, D$ are irreducible with $\dim D=\dim
X_{(e_i)}=\dim X_{(\tilde{e}_i)}$, then $X_{(e_i)}$ and
$X_{(\tilde{e}_i)}$ are birational.
\end{theorem}

\begin{proof}
We use the same notation as above. When $s=1$, then $X_{(e_i)} = X_{(\tilde{e}_i)}$, so nothing needs
to be proved. Now we assume $s \geq 2$.

\medskip

By Lemma \ref{decomposelattice} and \eqref{eq: L}, we have
\[
\Ann(e)=\Ann(e,\tilde{e}) \oplus L.
\] For any closed point $x \in X_{(e_i)}$, we can write \[x=(y,\omega) \in \spec
(\Cc[\Ann(e)])\] with $y \in \spec(\Cc[\Ann(e,\tilde{e})])$ and
$\omega \in \spec (\Cc[L])$. We claim that there exists
a morphism $\pi$:
\[
\pi : X_{(e_i)} \to D
\] defined by $x \mapsto y$.

\medskip

Indeed, by
the construction of $A^{(k)}(y)$ (c.f.\eqref{eq: matrix A}), the following matrix equation
\begin{equation}\label{eq: matrix eq}
\begin{pmatrix}
A^{(1)}(y)&&&\\
&A^{(2)}(y)&&\\
&&\ddots&\\
&&&A^{(r)}(y)\\
\end{pmatrix}
\begin{pmatrix}
\ww^{(1)}\\
\ww^{(2)}\\
\vdots\\
\ww^{(r)}\\
\end{pmatrix}
=0
\end{equation}
gives the variety $X_{({e_i})}$, where $\ww^{(k)}
=\left(1,\chi^{w^{(k)}_{2}},\cdots, \chi^{w^{(k)}_{n_{k}}}\right)^t$.

\medskip

Hence, for a closed point $(y,\omega) \in X_{({e_i})}$ and for all $k$, $A^{(k)}(y)\ww^{(k)} = 0$. Because $\ww^{(k)}\not\equiv0$,
we must have $\det(A^{(k)}(y))=0$. Hence, $y$ lives in
$D^{(k)}$ for all $ k$, and thus $y \in D = \cap_{k=1}^r D^{(k)}$. This shows that the
natural projection $\spec (\Cc[\Ann(e)]) \to \spec (\Cc[\Ann(e,
\tilde{e})])$  maps $X_{(e_i)}$ to $D$. We denote this morphism by
$\pi$.

\medskip

Next, we show that $\pi$ is generically injective, that is, $\pi$ is
injective on a nonempty open subset of $X_{(e_i)}$. Roughly
speaking, the proof rests on the fact that a Calabi-Yau variety
cannot be uniruled. We show that if $\pi$ is not generically
injective, then $X_{(e_i)}$ is a uniruled variety. However, it has a natural compactification $\overline{{X}_{(e_i)}}$
which is a projective, Calabi-Yau variety with canonical Gorenstein
singularities. Putting these facts together, we get a contradiction.
The details are as follows:

\medskip

Suppose $\pi$ is not generically injective. By a theorem of
Chevalley (\cite{Hartshorne} Chapter II Ex.3.22(e)), there exists a
nonempty open set $V \subset \pi(X_{(e_i)})$ such that over $V$, the
fibres have the same dimension $h$. Let $ y \in V$, and let
$(X_{(e_i)})_y$ be the fibre over $y$. 

\medskip

For $(y,\omega) \in (X_{(e_i)})_y$, we can write \[\omega = (\chi^{-w^{(1)}_2}, \ldots, \chi^{-w^{(1)}_{n_k}},\ldots, \chi^{-w^{(r)}_2},\ldots, \chi^{-w^{(r)}_{n_r}}).\] Then by the matrix equation \eqref{eq: matrix eq}, for general coefficients,
\[
\Omega_y = \{\omega \in \spec(\Cc[L]) \mid (y,\omega)
\in (X_{(e_i)})_y\}
\] is cut out by $s$ linear equations (linear with respect to $\omega$)
\begin{equation}\label{eq: linear equations}
c^{(k)}_{i}(y) + \sum_{2 \leq j \leq n_k} c^{(k)}_{i,j}(y) \cdot \chi^{-w^{(k)}_j}=0, \quad 1 \leq k \leq r, 1 \leq i \leq n_k, 
\end{equation} where the coefficients are regular functions in $y$
\[\begin{split}
&c^{(k)}_{i}(y) = \chi^{-u^{(k)}_{1}} g^{(k)}_{i,1}, \quad 1\leq i \leq n_k,\\
&c^{(k)}_{i,j}(y) = \chi^{-u^{(k)}_{i} - w^{(k)}_{j}} g^{(k)}_{i,j}, \quad 1 \leq i \leq n_k, 2 \leq j \leq n_k.
\end{split}\]

As a result, $\Omega_y$ contains a dense open set of an affine space determined by equations \eqref{eq: linear equations}. Thus for any $y \in V$, $\Omega_y$ either consists a single closed point or of positive dimension $h$. If the later happens, by solving the system of equations \eqref{eq: linear equations}, there exists a (non-unique) birational embedding $\Omega_y \hookrightarrow \Cc^h$ which is well defined in a Zariski open neighborhood $V_1 \subset V$ of $y$. In other words, we define a birational morphism
\[
\pi^{-1}(V_1) \to  V_1 \times \Pp^h.
\] This shows that $X_{(e_i)}$ is a ruled variety and, in particular,
a uniruled variety.

\medskip

In Appendix \ref{app} (Theorem \ref{complete intersection is CY} and
Remark \ref{compatification}), we show that the natural compactification
$\overline{X_{(e_i)}}$ of $X_{(e_i)}$ is a projective, Calabi-Yau variety with
canonical Gorenstein singularities. Let $\widetilde{X_{(e_i)}}$ be
a resolution of singularities of $\overline{X_{(e_i)}}$. It is also a uniruled
variety. Because $\overline{X_{(e_i)}}$ is a Calabi-Yau variety with
canonical singularities, the canonical divisor
$\widetilde{K_{(e_i)}}$ of $\widetilde{X_{(e_i)}}$ is
\[
\widetilde{K_{(e_i)}}= \sum_j c_j E_j,\quad c_j \geq 0,
\]
where $E_j$ are the exceptional divisors. Hence,
$H^0(\widetilde{X_{(e_i)}}, \Oo(\widetilde{K_{(e_i)}})) \neq 0$.
However, because $\widetilde{X_{(e_i)}}$ is a smooth, proper
uniruled variety over $\Cc$, we have $H^0(\widetilde{X_{(e_i)}},
\Oo(\widetilde{K_{(e_i)}})) = 0$ (\cite{Kollar95} IV Corollary
1.11). This is a contradiction, and hence $\pi$ is generically
injective.

\medskip

Let $U \subset X_{(e_i)}$ be an open set where $\pi|_U$ is
injective. Because for general coefficients, $X_{(e_i)}$ is smooth
of dimension $d-s$ (Proposition~\ref{GenericDeltaRegular}), $\pi(U)$
is a constructible subset of $D$ with dimension $d-s$. This is the
same dimension as $D$ by assumption. Thus $\pi$ is a dominant morphism as well.
By Lemma~\ref{birationality}, $X_{(e_i)}$ is birational to $D$.

\medskip

Because the construction is symmetric, a similar
argument can be used to show that $X_{(\tilde{e}_i)}$ is birational
to $D$. We sketch the argument below:

\medskip

First, by the proof of Lemma~\ref{decomposelattice}, one has a
decomposition of lattices
\[
\begin{split}
\Ann(\tilde{e})=&\Ann(e,\tilde{e}) \\
&\oplus \Zz[u^{(1)}_{2}-u^{(1)}_{1}]\oplus \dots \oplus \Zz[u^{(1)}_{n_1}-u^{(1)}_{1}] \\
&\oplus \cdots \\
&\oplus \Zz[u^{(r)}_{2}-u^{(r)}_{1}] \oplus \dots \oplus \Zz[u^{(r)}_{n_r}-u^{(r)}_{1}],
\end{split}
\]
where $u^{(k)}_{i} \in \overline{M} $ is the same as \eqref{eq: def of u}. Then $u^{(k)}_{i}-u^{(k)}_{1}$ can be viewed as an analogue of $w^{(k)}_{i}$ in \eqref{eq: decomposition of Ann(e)} because it satisfies the required relations  of
Lemma \ref{decomposelattice} with $e^{(k)}_{i}, \tilde{e}^{(k)}_{i}$
switched. Correspondingly, we have a decomposition of
tori (c.f. \eqref{eq: decomposition of toric}):
\[
\spec(\Cc[\Ann(\tilde{e})]) = \spec(\Cc[\Ann(e,\tilde{e})]) \times
(\Cc^*)^{s-r}.
\]
The determinantal variety can be defined similarly in $\spec(\Cc[\Ann(e,\tilde{e})])$ and is the same as $D$. There also exists a morphism $X_{(\tilde{e}_i)} \to D$ as before, and for
the same reason, it is a birational morphism.

\medskip

Hence $X_{(e_i)}$ and $X_{(\tilde{e}_i)}$ are both birational to
$D$, and this completes the proof.
\end{proof}

\begin{remark}\label{require det}
As explained in Remark \ref{rmk: irreducibility}, the irreducibility of $D$ is crucial for the theorem and we cannot take it for granted. However,  $X_{(e_i)}$ and $X_{(\tilde{e}_i)}$ are irreducible in practice because in the mirror construction, $\overline{X_{(e_i)}}, \overline{X_{(\tilde e_i)}}$ are always assumed to be irreducible. In fact, if $X_{(e_i)}$ is not irreducible, then it is a disjoint union of its irreducible components because of the toroidal singularity, but this will contradict the irreducibility assumption on $\overline{X_{(\tilde e_i)}}$. Besides, Theorem 3.3 in \cite{BB94} provides a combinatorial criteria for when this can be checked.
\end{remark}

\begin{remark}
It is reasonable to require that $\dim D = \dim X_{(e_i)} = d-s$.
Indeed $D= \cap_{i=1}^r D_i$ is a variety in
$\spec(\Cc[\Ann(e,\tilde{e})])$\ $ \cong (\Cc^*)^{d-(s-r)}$ defined
by the intersection of $r$ hypersurfaces. Thus $D$ is expected to
have dimension $d-s$ for generic choice of coefficients.
\end{remark}

\section{An application to the Calabrese-Thomas's example}{\label{application}}

We explain how to use Theorem \ref{main} (in fact Theorem \ref{baby version of main} is enough) to prove
the birationality of the Calabi-Yau threefolds for the second
example given by Calabrese and Thomas in \cite{CT14}\footnote{Their first example is explained in the setting of Clifford double mirrors \cite{BL16}.}. Once this is established, we can get their result of
derived equivalence for such Calabi-Yau varieties readily from the
main result of Bridgeland (\cite{Bridgeland}). In the process, one
will see that although we have to have some extra assumptions on
the determinantal variety (see
Remark \ref{rmk: irreducibility}), they are quite easy to check
in practice once the nef-partitions are known. 

\medskip

We first recall the construction of this example
following~\cite{CT14}, leaving its motivation and the proof of derived
equivalence from the homological projective dual perspective to the
interested readers.

\medskip

Let $\pi: P \to \Pp^5$ be a blowup of a closed point $0$ in the
complex projective space $\Pp^5$, with exceptional divisor $E$. One
considers the divisor in the linear system of \[K_{P}^{-1/2}:=
\pi^*\Oo(3)(-2E).\] Picking a generic pencil 
\begin{equation}\label{eq: pencil}
\Pp^1 \subset
|\pi^*\Oo(3)(-2E)|
\end{equation} of such divisors, and let 
\[X \subset P\] be
the base locus (i.e. $X$ is the complete intersection of two general
elements of $|\pi^*\Oo(3)(-2E)|$). By the adjunction formula, $X$ is
a Calabi-Yau $3$-fold. Let $Y \subset \Pp^4 \times \Pp^1$ be a
complete intersection of general elements of $(2,1)$ and $(3,1)$
divisors in $\Pp^4 \times \Pp^1$, which is another Calabi-Yau
$3$-fold. In \cite{CT14}, Calabrese and Thomas showed that the
universal hypersurface associated to the pencil \eqref{eq: pencil} is in fact the blowup
of $Y$ inside $\Pp^4 \times \Pp^1$. They gave two semi-orthogonal
decompositions of the derived category of this universal hypersuface
involved with $D^b(X), D^b(Y)$ respectively. By a sequence of
mutations, they could prove the derived equivalence of $X$ and $Y$.

\medskip

We only need $X$ and $Y$ in the sequel. We will show that they
are exactly the toric multiple mirrors in the Batyrev-Borisov
construction.

\medskip

Let $M: = \Zz^6/\Zz(1,1,1,1,1,1)$ be a lattice of rank $5$ with
$\bar{v}_i$ the image of the standard coordinate $v_i$ in $M$. The
dual lattice is \[N: = \{(x_1,\dots, x_6) \in \Zz^6 \mid
\sum_{i=1}^6 x_i = 0\}.\] The blowup of $\Pp^5$ at a closed point
corresponds to the fan $\Sigma(\Delta)$,  where $\Delta =
\Conv(\bar{v}_1, \bar{v}_2 ,\cdots, \bar{v}_6, -\bar{v}_6 )$. If
$E_i$ denotes the torus invariant divisor corresponding to the primitive element
$\bar{v}_i, 1\leq i \leq 6$, and $E$ denotes the exceptional divisor
corresponding to $-\bar{v}_6$, then
\[\pi^*\Oo(3)(-2E) \sim E_1 +E_2 +E_3 +E \sim E_4 + E_5 + E_6. \]
Thus, we have one nef-partition of $\Delta$, in this case
\[
\Delta_1 = \Conv (0, \bar{v}_1, \bar{v}_2,\bar{v}_3,-\bar{v}_6),
\quad \Delta_2 = \Conv (0,\bar{v}_4,\bar{v}_5,\bar{v}_6).
\]
The Gorenstein cone of this nef-partition is \[K = \Rr_{\geq
0}(1,0;\Delta_1) + \Rr_{\geq 0}(0,1;\Delta_2) \subset \Rr^2 \oplus
M_{\Rr}.\] We can write the degree element $(1,1;\0) \in \Zz^2 \oplus
M$ in two ways:
\[
(1,1;\0)= e_1 + e_2 = \tilde{e}_1 + \tilde{e}_2
\] where
\[
\begin{split}
&e_1 = (1,0;\0), \quad e_2 = (0,1;\0)\\
&\tilde{e}_1 = (1,0;-\bar{v}_6), \quad \tilde{e}_2 = (0,1;
\bar{v}_6).
\end{split}
\] As explained before, we have multiple mirrors $X_{(e_i)},
X_{(\tilde{e}_i)}$ and $X_{(e_i)}=X$. Next, we will show that
$X_{(\tilde{e}_i)} = Y$ where $Y$ is defined above in Calabrese and
Thomas' example. We do this by passing the decomposition
$(1,1;\0)= \tilde{e}_1 + \tilde{e}_2$ to the corresponding
nef-partition $\{\tilde{\Delta}_1, \tilde{\Delta}_2\}$. As in
Section~\ref{sec2}, we use $\{\nabla_1, \nabla_2\},
\{\tilde{\nabla}_1, \tilde{\nabla}_2\}$ to denote the dual
nef-partitions of $\{\Delta_1, \Delta_2\}, \{\tilde{\Delta}_1,
\tilde{\Delta}_2\}$ respectively.

\medskip

Because the vertices of $\nabla_1$ are piecewise linear functions on
the fan $\Sigma(\Delta_1 \cup \Delta_2)$, choosing $-1$ if the
primitive ray comes from $\Delta_1$ and $0$ if the primitive ray
comes from $\Delta_2$~\cite{Borisov93}, we can find the vertices of
$\nabla_1$ in $N$:
\[
\begin{array}{lll}
(0,0,0,0,0,0),&(1,-1,-1,0,0,1),&(-1,1,-1,0,0,1), \\
(-1,-1,1,0,0,1),&(-1,-1,-1,2,0,1),&(-1,-1,-1,0,2,1), \\
(2,-1,-1,0,0,0),&(-1,2,-1,0,0,0), &(-1,-1,2,0,0,0),\\
(-1,-1,-1,3,0,0), &(-1,-1,-1,0,3,0).&
\end{array}
\] Similarly, we can find the vertices of $\nabla_2$:
\[
\begin{array}{lll}
(0,0,0,0,0,0), &(2,0,0,-1,-1,0), &(0,2,0,-1,-1,0),\\
(0,0,2,-1,-1,0), &(0,0,0,1,-1,0), &(0,0,0,-1,1,0),\\
(3,0,0,-1,-1,-1), &(0,3,0,-1,-1,-1), &(0,0,3,-1,-1,-1),\\
(0,0,0,2,-1,-1),& (0,0,0,-1,2,-1).&
\end{array}
\] Then, the dual Gorenstein cone is $K^\vee=\Rr_{\geq 0}(1,0;\nabla_1) + \Rr_{\geq 0}(0,1;\nabla_2)$.
We can use $\{\tilde{e}_1, \tilde{e}_2\}$ to obtain the vertices of
$\tilde{\nabla}_1$ :
\[
\begin{array}{lll}
(0,0,0,0,0,0), &(2,-1,-1,0,0,0),&(-1,2,-1,0,0,0),\\
(-1,-1,2,0,0,0),& (-1,-1,-1,3,0,0), &(-1,-1,-1,0,3,0),\\
(3,0,0,-1,-1,-1), &(0,3,0,-1,-1,-1), &(0,0,3,-1,-1,-1),\\
(0,0,0,2,-1,-1), &(0,0,0,-1,2,-1).&
\end{array}
\] And the vertices of $\tilde{\nabla}_2$ are:
\[
\begin{array}{lll}
(0,0,0,0,0,0), &(1,-1,-1,0,0,1),&(-1,1,-1,0,0,1),\\
(-1,-1,1,0,0,1), &(-1,-1,-1,2,0,1),&(-1,-1,-1,0,2,1),\\
(2,0,0,-1,-1,0),&(0,2,0,-1,-1,0), &(0,0,2,-1,-1,0),\\
(0,0,0,1,-1,0), &(0,0,0,-1,1,0).&
\end{array}
\]
From this we can find the vertices of $\tilde{\Delta}_1$:
\[
\begin{array}{ll}
\0, \quad t_1:= -\bar{v}_1  -\bar{v}_2  -\bar{v}_3  -\bar{v}_4  -\bar{v}_5, &  w_1 := -\bar{v}_1  -\bar{v}_2  -\bar{v}_4  -\bar{v}_5, \\
w_2:= -\bar{v}_1 -\bar{v}_3  -\bar{v}_4 -\bar{v}_5, & w_3:=
-\bar{v}_2 -\bar{v}_3  -\bar{v}_4  -\bar{v}_5.
\end{array}
\] And similarly, the vertices of $\tilde{\Delta}_2$ are:
\[
\begin{array}{ll}
\0, \quad t_2:=\bar{v}_1  +\bar{v}_2  +\bar{v}_3  +\bar{v}_4  +\bar{v}_5,&w_4: = \bar{v}_1  +\bar{v}_2  +\bar{v}_3  +\bar{v}_4  +
2\bar{v}_5,\\
 w_5:=\bar{v}_1  +\bar{v}_2  +\bar{v}_3  +2\bar{v}_4
+\bar{v}_5.&
\end{array}
\] We observe that $\{t_1, w_1, w_2, w_3, w_4\}$ forms a basis of $\Zz^5$,
and satisfy
\[t_1 + t_2 = 0, \quad \sum_{i=1}^4 w_i = -w_5. \]
Hence, the toric variety defined by the fan $\Sigma(\Conv(\tilde{\Delta}_1
\cup \tilde{\Delta}_2))$ is $\Pp^1 \times \Pp^4$. Moreover, the
divisor corresponding to $\tilde{\Delta}_1$ is a $(1, 3)$ divisor
and the divisor corresponding to $\tilde{\Delta}_2$ is a $(1, 2)$
divisor in $\Pp^1 \times \Pp^4$. The complete intersection defined
by this new nef-partition is exactly the $Y$ in Calabrese and
Thomas' example. Therefore, we have reconstructed the $X, Y$ by realizing
them as a toric multiple mirrors $X_{(e_i)}(= X),
~X_{(\tilde{e}_i)}(= Y)$.

\medskip

By Theorem \ref{main} (or Theorem \ref{baby version of main}), in order to show that $X,Y$ are birational, we
only need to check that the determinantal variety $D$ is irreducible and of dimension $3$.

\medskip

Observing that $\dim(\spa_\Rr \{e_1, e_2, \tilde{e}_1,
\tilde{e}_2\}) = 3$, and $S_{i,j} \neq \emptyset$ for any $1 \leq i,
j \leq 2$ (see \eqref{eq: polytope S_i,j} for notation), we can
conclude that $D$ satisfies these conditions. In fact,
$\dim(\spa_\Rr \{e_1, e_2, \tilde{e}_1, \tilde{e}_2\}) = 3$ implies
that $r=1$, hence we only use a single matrix to define $D$. Then $S_{i,j} \neq
\emptyset$ implies that the $(i,j)$-entry of this matrix \eqref{eq: 2 by 2 matrix after normalization} is
non-zero. Because the coefficients of each entry are chosen
generally, $D$ must be an irreducible hypersurface in $(\Cc^*)^4$.

\medskip

By Theorem \ref{main}, we conclude that $X, Y$ are birational (see also \cite{CT14} Proposition 4.3).

\medskip

After establishing the birationality, because $X, Y$ are
smooth Calabi-Yau 3-folds, we can conclude that they are derived
equivalent by Bridgeland's result (\cite{Bridgeland} Theorem 1.1).

\medskip

More generally, Batyrev and Nill conjectured that
\begin{conjecture}[\cite{BN08} Conjecture 5.3]
There exists an equivalence (of Fourier-Mukai type) between the
derived categories of coherent sheaves on the two Calabi-Yau complete
intersections $\overline{X_{(e_i)}}$ and
$\overline{X_{(\tilde{e}_i)}}$.
\end{conjecture}

Favero and Kelly give an affirmative answer to the above conjecture for smooth DM-stacks associated to $\overline{X_{(e_i)}}$, $\overline{X_{(\tilde{e}_i)}}$ (\cite{FK14} Theorem 6.3). In particular, this provides another line of proof for the derived equivalence of the above example.

\section{Appendix: $\Delta$-regularity, singularities and Calabi-Yau varieties}\label{app}

Roughly speaking, $\Delta$-regularity is a condition on the
smoothness of stratifications with correct dimensions. In this
Appendix, we generalize the concept of
$\Delta$-regularity \cite{Batyrev93, Batyrev94} of a
hypersurface to an intersection of several hypersurfaces in toric
varieties. We will show that for general coefficients, the complete
intersections defined by a nef-partition are $\Delta$-regular. Under the $\Delta$-regular assumption, the
singularities of the complete intersection are inherited from the
ambient toric variety. Using these results, we will show that an
irreducible $\Delta$-regular complete intersection associated to a
nef-partition is a Calabi-Yau variety with canonical Gorenstein
singularities. This fact is used in the proof of Theorem \ref{main}
by showing that the morphism $\pi$ is generically injective.

\medskip

Let $\Sigma \subset N_\Rr$ be a fan, and $X(\Sigma)$ be the toric
variety defined by $\Sigma$. If $\sigma \in \Sigma$ is a cone, let
$T_\sigma$ be the torus corresponding to $\sigma$. Then we have the
following stratification:
\[
X(\Sigma) = \bigcup_{\sigma \in \Sigma} T_\sigma\ .
\]

\begin{definition}\label{regular}
Let $V_i, 1 \leq i \leq s,$ be hypersurfaces of $X(\Sigma)$, and let
$V = \bigcap_{i=1}^s V_i$ be the scheme-theoretic intersection. Then
$V$ is called \emph{$\Delta$-regular} if $V$ is equidimensional
and for all $\sigma \in \Sigma,~ T_{\sigma} \cap V $ is either empty
or smooth of codimension $s$ in $T_{\sigma}$.
\end{definition}

We use the name $\Delta$-regularity following
Batyrev~\cite{Batyrev93, Batyrev94}, where $\Delta$ is a
polytope, and the regularity is about a hypersurface defined by a
Laurent polynomial with Newton polytope inside $\Delta$.

\begin{remark}\label{other interpretation of delta regular}
The $\Delta$-regular condition requires the linear independence of
the cotangent spaces at a common intersection point. This takes care
of both smoothness and codimension.
\end{remark}

Recall that for a nef-partition $\{\Delta_i \mid 1 \leq i \leq s\}$ and its dual $\{\nabla_i \mid 1 \leq i \leq s\}$, we defined in \eqref{eq: Sigma(nabla)} the fan $\Sigma(\nabla)$ with $\nabla=\Conv(\cup_{i=1}^s \nabla_i)$, and the associated toric variety $X(\Sigma(\nabla))$. In  \eqref{eq: sheaf G} we defined the toric invariant Cartier divisor $\mathcal{G}_i$ associated to $\nabla_i$, whose global sections can be identified in the following way
\[
H^0(X(\Sigma(\nabla)), \mathcal{G}_i) \cong \{ \sum_{v \in l(\Delta_i)} c_v
\chi^v \mid c_v \in \Cc \}.
\]
As in \eqref{eq: overline X}, $\{\overline{X_{(\Delta_i)}}\}$ is a family of subschemes of $X(\Sigma(\nabla))$ which are scheme-theoretic intersections of zero loci $V_{f_i}$ of $f_i$ with respect to $\mathcal{G}_i$:
\[
\{\overline{X_{(\Delta_i)}} =\bigcap_{i=1}^s V_{f_i} \mid f_i \in H^0(X(\Sigma(\nabla)),
\mathcal{G}_i) \}.
\]

To show that a general
member of this family is $\Delta$-regular, we first show that they satisfy the requirement on the codimension for each
$T_\sigma$.

\begin{proposition}\label{GeneralSmoothStratification}
For general coefficients $c_v \in \Cc$ of $f_i =  \sum_{v \in
l(\Delta_i) } c_v \chi^v$, $1 \leq i \leq s$, the scheme $T_\sigma
\bigcap \overline{X_{(\Delta_i)}} =T_\sigma \bigcap \left(\cap_{i=1}^s V_{f_i}\right) $ is
either empty or smooth of codimension $s$ for every $T_{\sigma}$.
\end{proposition}

\begin{proof}
Because nefness and base point
freeness are equivalent on toric varieties, the linear system
$|\mathcal{G}_i|$ is base point free.

\medskip

Next, we use a similar argument in the proof of Bertini's theorem (\cite{Hartshorne} III
Corollary 10.9 and Remark 10.9.2) to show that for general
coefficients, either $T_\sigma \cap \overline{X_{(\Delta_i)}}$ is empty or smooth of
codimension $s$, where $\sigma \in \Sigma(\nabla)$. If the dimension of the
linear system $|\mathcal{G}_i|$ is $n_i$, then these linear systems altogether define a
morphism
\[
f: T_\sigma \hookrightarrow X(\Sigma(\nabla)) \to \Pp^{n_1 } \times
\cdots \times \Pp^{n_s}.
\]

Let $\Pp:=\Pp^{n_1 } \times \cdots \times \Pp^{n_s}$, and we
consider it as a homogeneous space under the action of $G:=
\operatorname{PGL}(n_1) \times \cdots \times
\operatorname{PGL}(n_s)$. Let $H_i \to \Pp^{n_i}$ be the inclusion
of a hyperplane $H_i \cong \Pp^{n_i -1}$, and
\[
g: H_1 \times \cdots \times H_s \to \Pp^{n_1} \times \cdots \times
\Pp^{n_s}
\]
be the product of these inclusions.

\medskip

We set $H:= H_1 \times \cdots \times H_s $, and for $\tau \in
G$, let $H^\tau$ be $H$ with the morphism $\tau \circ g$ to $\Pp$.
We can apply Kleiman's theorem (\cite{Hartshorne}~III Theorem 10.8)
to $g$ and conclude that there exists a nonempty open set $W \subset
G$, such that for all $\tau \in W$, $T_\sigma \times_\Pp H^\tau$
is nonsingular and either empty or of codimension $s$. However, one
can show that $f^{-1}(H^\tau)$ is exactly the scheme theoretic
intersection $T_\sigma \cap \overline{X_{(\Delta_i)}}$ defined by the linear systems
$|\mathcal{G}_i|, 1 \leq i \leq s$. This completes the proof.
\end{proof}

Recall that toric Gorenstein, canonical and terminal
singularities are characterized by the combinatorial properties of
cones \cite{Reid} (See also \cite{Batyrev93}):

\begin{proposition}\label{ToricSingularity}
Let $n_1, \ldots, n_r \in N$ be primitive integral generators of all
$1$-dimensional faces of a cone $\sigma \subset N_\Rr$.

\begin{enumerate}
\item $U_{\sigma}$ has Gorenstein singularity if and only if $n_1, \ldots, n_r$
are contained in an affine hyperplane
\[ H_{\sigma} := \{ y \in N_{\Rr} \mid \langle
k_{\sigma}, y \rangle = 1 \},
\] for some $k_{\sigma} \in M$.

\item Assume $U_{\sigma}$ has Gorenstein singularity, then it has
canonical singularity if and only if \[ N \cap \sigma \cap \{y \in
N_{\Rr} \mid \langle  k_{\sigma} , y \rangle <  1 \} = \{ \0 \}. \]

\item Assume $U_{\sigma}$ has Gorenstein singularity, then it has terminal singularity if and only if
\[ N \cap \sigma \cap \{y \in N_{\Rr} \mid \langle  k_{\sigma},
y \rangle \leq  1 \} = \{ \0, n_1, \ldots, n_r \}. \]
\end{enumerate}
\end{proposition}

\begin{remark}\label{singularity of X(Sigma(nabla))}
Because the only interior lattice point of $\nabla$ is $\0$, we know that the toric variety $X(\Sigma(\nabla))$ has Gorenstein canonical singularities by Proposition \ref{ToricSingularity}.
\end{remark}

\begin{theorem}\label{singularity}
Suppose $X(\Sigma(\nabla))$ has Gorenstein canonical (resp.
terminal) singularities, and each irreducible component of $\overline{X_{(\Delta_i)}}$ satisfies $\Delta$-regularity, then $\overline{X_{(\Delta_i)}}$ is a normal variety with Gorenstein canonical (resp. terminal) singularities.
\end{theorem}

\begin{proof}
Let $U_{\sigma, N}$ be the toric variety associated to
the cone $\sigma$ in the lattice $N$, and $U_{\sigma, N(\sigma)}$ be the toric variety associated to $\sigma$ in the lattice $N(\sigma):=N\cap \Rr \sigma$. Let $\rk N =n, \rk N(\sigma)=l$, then we have
\[
U_{\sigma, N} \cong U_{\sigma, N(\sigma)} \times (\Cc^*)^{n-l}.
\]

Under this identification, $T_\sigma \cong p_\sigma \times
(\Cc^*)^{n-l}$, where $p_\sigma \in U_{\sigma, N(\sigma)}$ is the
unique torus invariant point. Let $f_1,\dots,f_s$ be the restrictions
of Laurent polynomials on $U_{\sigma, N}$. In particular, they are
analytic function on $U_{\sigma, N} \cong U_{\sigma, N(\sigma)}
\times (\Cc^*)^{n-l}$. By $\Delta$-regularity, for any $(p_\sigma
;a_1,\dots,a_{n-l}) \in T_{\sigma} \cap V_{f_1} \cap \cdots \cap
V_{f_s}$, the Jacobian matrix
\[
\left( \frac{\partial g_i}{\partial t_j}(p_\sigma
;a_1,\dots,a_{n-l})\right)_{ij},\quad 1\leq i \leq s, 1\leq j \leq
n-l
\]
has rank $s=\dim (T_{\sigma}) - \dim (T_{\sigma} \cap \overline{X_{(\Delta_i)}})$. By
continuity, in an analytic neighborhood of $(p_\sigma
;a_1,\dots,a_{n-l}) \in U_{\sigma, N(\sigma)} \times (\Cc^*)^{n-l}$,
the matrix
\[
\left( \frac{\partial g_i}{\partial t_j}(\ {x}\ ;\
{t}\ )\right)_{ij},\quad 1\leq i \leq s, 1\leq j \leq n-l
\]
has rank $s$, where ${x} \in U_{\sigma, N(\sigma)}$ in a
neighborhood of $p_\sigma$, and ${t}:=
(t_1,\cdots,t_{n-l}) \in (\Cc^*)^{n-l}$ in a neighborhood of
$(a_1,\cdots,a_{n-l})$.

\medskip

Without loss of generality, we can assume the $s \times s$ minor
with $1 \leq i \leq s, 1 \leq j \leq s$ is nonvanishing. Thus,
we can apply the implicit function theorem to $f_1,\dots,f_s$ and show that there are $s$ analytic functions $u_1,\dots,u_{s}$
defined on an open neighborhood of $(p_\sigma;a_{s+1},\dots,a_{n-l})
\in U_{\sigma, N(\sigma)} \times (\Cc^*)^{n-l-s} $. Moreover, if
${t}':=(t_{s+1},\cdots t_{n-l})$, then any point in a
neighborhood of $(p_\sigma;a_{1},\dots,a_{n-l}) \in U_{\sigma,
N(\sigma)} \times (\Cc^*)^{n-l}$ satisfying $f_1=\cdots=f_s=0$
can be written as
\[
(\ {x}\ ; \ u_1( {x},  {t}'), \cdots,
u_s( {x}, {t}')\ ; \  {t}').
\] This shows that a neighborhood of $(p_\sigma; a_1,
\cdots, a_{n-l}) \in T_{\sigma} \cap \overline{X_{(\Delta_i)}}$ is locally analytically
isomorphic to a product of a neighborhood of $p_\sigma$ in
$U_{\sigma, N(\sigma)}$ with a neighborhood of
$(a_{s+1},\dots,a_{n-l})$ in $(\Cc^*)^{n-l-s}$.

\medskip

From the above isomorphism, we know that $\overline{X_{(\Delta_i)}}$ is normal because $U_{\sigma,
N(\sigma)} \times (\Cc^*)^{n-l}$ is normal and normality is preserved under
an analytic isomorphism. Moreover, the Gorenstein singularity is also a locally analytic property. In fact, the completion of the local ring of a variety is the same as
the completion of the local ring of its analytic space, and a local ring is Gorenstein if and only if its
completion is Gorenstein. Likewise, canonical and terminal
singularities are both local analytic properties 
(\cite{Matsuki}~Proposition~4-4-4). Thus we have proved the theorem.
\end{proof}

\begin{remark}
In general, $\overline{X_{(\Delta_i)}}$ may have singularities. However, it is shown in \cite{Batyrev94} Theorem 4.2.2 that there also exists a crepant partial desingularization (called MPCP-desingularization) of $\overline{X_{(\Delta_i)}}$ with only $\Qq$-factorial terminal singularities. Moreover, when $\dim \overline{X_{(\Delta_i)}}\leq 3$, such partial desingularization is a smooth Calabi-Yau variety (\cite{Batyrev94} Corollary 4.2.3).
\end{remark}

The combination of Proposition \ref{GeneralSmoothStratification} and Theorem \ref{singularity} shows that the general members in $\{\overline{X_{(\Delta_i)}}\}$ are indeed $\Delta$-regular.

\begin{proposition}\label{GenericDeltaRegular}
For general coefficients, $\overline{X_{(\Delta_i)}} =\bigcap_{i=1}^s V_{f_i}$ is a
$\Delta$-regular intersection. It has finite disjoint irreducible
components, and each irreducible component has canonical
Gorenstein singularities.
\end{proposition}
\begin{proof}
By Theorem \ref{singularity}, $\overline{X_{(\Delta_i)}}$ is normal, hence its irreducible components are disjoint. By the argument of Theorem \ref{singularity}, we know that for any point on an irreducible component, there exists a neighborhood
locally analytically isomorphism to an open neighborhood of
dimension $n-s$. This justifies the equidimensional requirement for $\Delta$-regularity. The claim for the singularities follows from  Remark \ref{singularity of X(Sigma(nabla))} and Theorem \ref{singularity}.
\end{proof}

\begin{remark}
From the above argument, one can show that for general
coefficients, and for any subset $I \subset \{1,2,\cdots,s\}$, the
scheme-theoretic intersection $\bigcap_{i \in I} V_{f_i}$ is also
$\Delta$-regular.
\end{remark}

By Proposition \ref{GenericDeltaRegular}, we assume $\overline{X_{(\Delta_i)}}$ to be a
$\Delta$-regular intersection associated to a nef-partition. First
recall following the proposition about the adjunction formula on a
Cohen-Macaulay scheme (\cite{KollarMori} Proposition 5.73).

\begin{proposition}\label{adjunction}
Let $P$ be a projective Cohen-Macaulay scheme of pure dimension $n$
over a field $k$, and $D \subset P$ an effective Cartier divisor.
Then $\omega_D \cong \omega_P(D) \otimes \mathcal{O}_D$. Here
$\omega_D, \omega_P$ are dualizing sheaves of $D, P$ respectively.
\end{proposition}

Applying this result and combining with Theorem \ref{singularity},
we have the following proposition.

\begin{theorem}\label{complete intersection is CY}
If a general $\overline{X_{(\Delta_i)}}$ is irreducible, then it is a Calabi-Yau
variety (that is, the canonical divisor is trivial) with canonical singularities.
\end{theorem}

\begin{proof}
Let us set $P:=X(\Sigma(\nabla))$ and $\overline X = \overline{X_{(\Delta_i)}}$. Because $P$ is a
Cohen-Macaulay scheme with at worst Gorenstein canonical singularities, we can apply Proposition \ref{adjunction} to
get
\[
\omega_{\overline{X}} \cong \omega_{P}(V_{f_1}+V_{f_2}+\cdots+V_{f_s}) \otimes
\mathcal{O}_{\overline X }.
\] By nef-partition, we have
$\Oo_P(-K_P) \cong \otimes_{i=1}^s
\Oo_P(V_{f_i})$. Therefore
\[
\omega_{P}(V_{f_1}+V_{f_2}+\cdots+V_{f_s}) \cong \mathcal{O}_P(-K_{P} +
V_{f_1}+V_{f_2}+\cdots+V_{f_s} ) \cong \mathcal{O}_P,
\] and hence $\omega_{\overline X } \cong
\Oo_{\overline X }$. On a normal variety, the dualizing sheaf is equivalent to
the canonical sheaf (\cite{KollarMori} Proposition 5.77). Using the
fact that $\overline X $ is a normal variety, we have $K_{\overline X } = 0$. This shows
that $\overline X $ is a Calabi-Yau variety.

\medskip

Just as in Proposition \ref{GenericDeltaRegular}, the claim for the singularities follows from Remark \ref{singularity of X(Sigma(nabla))} and Theorem \ref{singularity}.
\end{proof}

\begin{remark}\label{compatification}
If the nef-partition $\{\Delta_i \mid 1 \leq i \leq s\}$ comes from
$\deg^\vee = \sum_{i=1}^s e_i$ as in Section \ref{sec3}, then $\overline{X_{(\Delta_i)}} \cap
(\Cc^*)^d = X_{(\Delta_i)} =  X_{(e_i)} $. In other words, $\overline{X_{(\Delta_i)}}$ is a natural projective
compactification of $X_{(e_i)}$, and we denote it by
$\overline{X_{(e_i)}}$ as in the proof of Theorem \ref{main}.
\end{remark}

\end{document}